\documentclass[english,12pt,a4paper,reqno,lenqo]{amsart}
\usepackage{amsfonts,amsmath,amsthm,amsfonts,amssymb,amstext}
\usepackage{mathrsfs}
\usepackage[leqno]{amsmath}
\usepackage{a4wide}
\usepackage[ansinew]{inputenc}
\usepackage{hyperref}
\usepackage{graphicx}
\usepackage[ansinew]{inputenc}
\usepackage{a4wide}
\newcommand{\leb}{\operatorname{Leb}}

\newcommand{\dist}{\operatorname{dist}}
\newcommand{\diam}{\operatorname{diam}}

\begin{document}
\newcommand{\mcup}{\mbox{$\bigcup$}}
\newcommand{\mcap}{\mbox{$\bigcap$}}

\def \RR {{\mathbb R}}
\def \ZZ {{\mathbb Z}}
\def \NN {{\mathbb N}}
\def \PP {{\mathbb P}}
\def \TT {{\mathbb T}}
\def \II {{\mathbb I}}
\def \JJ {{\mathbb J}}

\def \vare {\varepsilon }

 \def \cf {\mathcal{F}}
 \def \cm {\mathcal{M}}
 \def \cn {\mathcal{N}}
 \def \cq {\mathcal{Q}}
 \def \cp {\mathcal{P}}
 \def \cb {\mathcal{B}}
 \def \cc {\mathcal{C}}
 \def \cs {\mathcal{S}}
 \def \bc {\mathcal{B}}
 \def \hc {\mathcal{C}}

\newcommand{\dem}{\begin{proof}}
\newcommand{\cqd}{\end{proof}}

\newcommand{\qand}{\quad\text{and}\quad}

\newtheorem{theorem}{Theorem}
\newtheorem{corollary}{Corollary}

\newtheorem*{Maintheorem}{Main Theorem}
\newtheorem*{Theorem*}{Theorem}

\newtheorem{maintheorem}{Theorem}
\renewcommand{\themaintheorem}{\Alph{maintheorem}}
\newcommand{\cmt}{\begin{maintheorem}}
\newcommand{\fmt}{\end{maintheorem}}

\newtheorem{maincorollary}[maintheorem]{Corollary}
\renewcommand{\themaintheorem}{\Alph{maintheorem}}
\newcommand{\cmc}{\begin{maincorollary}}
\newcommand{\fmc}{\end{maincorollary}}

\newtheorem{T}{Theorem}[section]
\newcommand{\ct}{\begin{T}}
\newcommand{\ft}{\end{T}}

\newtheorem{Corollary}[T]{Corollary}
\newcommand{\cco}{\begin{Corollary}}
\newcommand{\fco}{\end{Corollary}}

\newtheorem{Proposition}[T]{Proposition}
\newcommand{\cpr}{\begin{Proposition}}
\newcommand{\fpr}{\end{Proposition}}

\newtheorem{Lemma}[T]{Lemma}
\newcommand{\cle}{\begin{Lemma}}
\newcommand{\fle}{\end{Lemma}}

\newtheorem{Sublemma}[T]{Sublemma}
\newcommand{\csle}{\begin{Sublemma}}
\newcommand{\fsle}{\end{Sublemma}}

\newtheorem{Question}{Question}

\newtheorem*{Conjecture}{Conjecture}

\theoremstyle{definition}

\newtheorem{Remark}[T]{Remark}
\newcommand{\cre}{\begin{Remark}}
\newcommand{\fre}{\end{Remark}}

\newtheorem{Definition}[T]{Definition}
\newcommand{\cd}{\begin{Definition}}
\newcommand{\fd}{\end{Definition}}

\allowdisplaybreaks

\title[GMY structures with (stretched) exponential tail]{Gibbs-Markov-Young structures with (stretched) exponential tail for partially hyperbolic attractors}

\author{Jos\'e F. Alves}
\address{Jos\'e F. Alves\\ Departamento de Matem\'atica, Faculdade  de Ci\^encias da Universidade do Porto\\
Rua do Campo Alegre 687, 4169-007 Porto, Portugal}
\email{jfalves@fc.up.pt} \urladdr{http://www.fc.up.pt/cmup/jfalves}

\author{Xin Li}
\address{Xin Li\\ Department of Mathematics, Soochow University, Suzhou 215006, Jiangsu, PR China;
and 
Departamento de Matem\'atica, Faculdade de Ci\^encias da Universidade do Porto\\
Rua do Campo Alegre 687, 4169-007 Porto, Portugal}
\email{lxinporto@gmail.com}

\date{\today}

\subjclass[2000]{37A05, 37C40, 37D25}

\keywords{Partially hyperbolic attractor, Gibbs-Markov-Young
structure, recurrence times, decay of correlations, large deviations}

\thanks{JFA was partially supported by Funda\c c\~ao Calouste Gulbenkian, by CMUP, by the European Regional Development Fund through the Programme COMPETE and by  FCT under the projects PTDC/MAT/099493/2008 and PEst-C/MAT/UI0144/2011. XL was supported by FCT}

\begin{abstract} In this work we extend the results obtained by  Gou\"{e}zel in \cite{G} to partially hyperbolic attractors. 
We study a forward invariant set $K$ on a Riemannian manifold
$M$ whose tangent space splits as dominated decomposition $T_K M=E^{cu}\oplus E^{s}$, for
which the center-unstable direction $E^{cu}$ is non-uniformly expanding
on some local unstable disk.
 We prove that the (stretched) exponential decay of recurrence times for an induced scheme can be deduced under the assumption of  (stretched) exponential decay of the time that typical points need to
achieve some uniform expanding in the center-unstable
direction.  As an application of our results we obtain  exponential decay of correlations and exponential large deviations for a class of partially hyperbolic diffeomorphisms  considered in~\cite{ABV}.

\end{abstract}

\maketitle

\setcounter{tocdepth}{2}

\tableofcontents


\section{Introduction}

 In the late 60's and beginning of 70's, Sinai,
Ruelle and Bowen brought Markov partitions and symbolic dynamics
into the theory of uniformly hyperbolic dynamics to prove the existence of the so-called \emph{Sinai-Ruelle-Bowen (SRB) measures} for these  systems; see \cite{Sin68,Bow70,Rue79}. According to Ruelle \cite[Preface]{Bow75}, ``this allowed the powerful techniques and results of statistical mechanics to be applied into smooth dynamics''. 

To study systems beyond those
uniformly hyperbolic, in \cite{Y1,Y2} Young  used some type of Markov partitions with infinitely many symbols to build towers for systems with nonuniform hyperbolic behavior, including Axiom A attractors, piecewise hyperbolic maps, billiards with convex scatterers, logistic maps, intermittent maps and H\'{e}non-type
attractors. Using  towers, Young  studied some statistical properties of these nonuniformly hyperbolic systems,
including the existence of SRB measures, exponential decay of correlations and the validity of the Central Limit Theorem for the SRB measure.  Roughly speaking, Young towers are characterized by some region of the phase space partitioned into an at most countable number of subsets with associated \emph{recurrence times}. Young called it a \emph{horseshoe with infinitely many branches}. These structures have some properties which address to Gibbs states and for that reason they are nowadays commonly  referred to as \emph{Gibbs-Markov-Young (GMY) structures}.

 In \cite{BV}, Bonatti and Viana considered partially hyperbolic attractors with mostly contracting central direction, meaning that the tangent bundle splits as $E^{cs}\oplus E^u$, with the  $E^u$ direction being uniformly expanding and the $E^{cs} $ direction having negative Lyapunov exponents.
They gave sufficient conditions for  the existence of SRB measures under those conditions. In \cite{Cas}, Castro showed the existence of GMY structures for such systems, thus obtaining statistical properties like exponential decay of correlations and the validity of the Central Limit Theorem. The Central Limit
Theorem   had also been obtained by Dolgopyat in \cite{D}.

However, as most of the richness of the dynamics in partially hyperbolic attractors appears in the  unstable direction,  the case $E^{cu}\oplus E^s$ (now with the stable direction being uniform and the unstable one being nonuniform)  comprises more difficulties than the case $E^{cs}\oplus E^u$. 
The existence of SRB's for some classes of non-uniformly hyperbolic systems  has been proved in~\cite{ABV} by Alves, Bonatti and Viana, both for non-uniformly expanding maps, in the non-invertible case, and for partially hyperbolic attractors of the type $E^{cu}\oplus E^s$, in the invertible case.
 In  the non-invertible case,  Alves, Luzzatto and Pinheiro  proved in \cite{ALP} the existence of GMY structures  of non-uniformly expanding maps. 
 Their approach,  originated from \cite{Y1} for Axiom A attractors, has shown to be not efficient enough to estimate the tail of recurrence times for non-uniformly hyperbolic systems with exponential or stretched exponential tail of hyperbolic  times. This is due to the fact that at each step of their algorithmic construction just a definite fraction of  hyperbolic times is used to construct new elements in the partition. In the invertible case, using arguments similar to those in~\cite{ALP},
Alves and Pinheiro in~\cite{AP3}  obtained GMY structures for partially hyperbolic attractors. Again, they only managed to prove the polynomial case: if the lack of expansion of the system at time $n$  in the center-unstable direction is polynomially small,  then the system has some GMY structure with polynomial decay of recurrence times.

 Gou\"{e}zel
developed a new construction in \cite{G} with more efficient control for the tail of the recurrence times in the non-invertible setting. As a starting point, Gou\"{e}zel used the fact that the attractor could be partitioned into a finite number of sets with small
size.
That gave rise to more precise
estimates than those  in \cite{ALP}, yielding also the (streched) exponential case for non-uniformly expanding maps. However,  for important combinatorial reasons, Gou\"{e}zel
 strategy could not be generalized  to the
partially hyperbolic setting $E^{cu}\oplus E^s$, in particular because the attractor is typically made of unstable leaves, which are not bounded in their intrinsic distance.
%
Partly inspired by \cite{G,P}, Alves, Dias and Luzzatto gave in \cite{ADL}  an improved \emph{local} GMY structure, more efficient than \cite{ALP} in the use of hyperbolic times, which made it possible to prove the integrability of recurrence times under very general conditions.

%

The main goal of this work is to fill a gap in the theory of partially hyperbolic diffeomorphisms of the type $E^{cu}\oplus E^s$, where, after \cite{AP3},  GMY structures are only known  with polynomial tail of recurrence times.
From these structures we get (stretched) exponential
Decay of Correlations and exponential Large Deviations for the systems under consideration, by  related
results in \cite{Y1,AP2,MN1}.
Our strategy is based in a mixture of techniques from \cite{ADL}
and \cite{G} and we construct a GMY structure by a method similar to \cite{ADL}, where   recurrence times were only proved to be integrable. To improve the efficiency of the algorithm in \cite{AP3}, our
method has a main difference, namely,  we keep track of all points with hyperbolic times at a given iterate and not just of a proportion of those points.



\subsection{Gibbs-Markov-Young structures}
 \label{se.hypstructures}  Here we recall the structures which have been introduced in
 \cite{Y1}.
Let $f:M\to M$ be a $C^{1+}$ diffeomorphism of a finite dimensional
Riemannian manifold~$M$,  \( \leb \) (Lebesgue measure) the
normalized Riemannian volume on the Borel sets of \( M \). Given a submanifold $\gamma\subset M$, we use $\leb_\gamma$
to denote the Lebesgue measure on $\gamma$, induced by the restriction
of the Riemannian structure to~$\gamma$.

\cd An embedded disk $\gamma\subset M$ is called a local {\em unstable manifold}
 if for all $x,y\in\gamma$
\begin{equation*}
  \dist(f^{-n}(x),f^{-n}(y))\to0,\quad\text{as
$n\to\infty$.}
\end{equation*}
Similarly, $\gamma$ is called
a {\em local stable manifold} if for all $x,y\in\gamma$
\begin{equation*}
  \dist(f^n(x),f^n(y))\to0,\quad\text{as $n\to\infty$.}
\end{equation*}
\fd


\begin{Definition} Given $n\ge1$, let $D^u$ be a unit disk in $\mathbb{R}^n$ and let $\text{Emb}^1(D^u,M)$ be the space of $C^1$ embeddings from
$D^u$ into $M$. A \emph{continuous family of $C^1$ unstable manifolds} is a set $\Gamma^u$ of unstable disks $\gamma^u$ satisfying the following properties: there is a
compact set $K^s$ and a
map $\Phi^u\colon K^s\times D^u\to M$ such that
\begin{enumerate}
\item $\gamma^u=\Phi^u(\{x\}\times D^u)$ is a local unstable
manifold;
\item $\Phi^u$ maps $K^s\times D^u$ homeomorphically onto its
image;
\item $x\mapsto \Phi^u\vert(\{x\}\times D^u)$ is a
continuous map from $K^s$ to $\text{Emb}^1(D^u,M)$.
\end{enumerate}
Continuous families of $C^1$ stable manifolds are defined analogously.
\end{Definition}

\begin{Definition}
A subset $\Lambda\subset M$ has a \emph{product
structure} if, for some $n\ge1$, there exist a continuous family of $n$-dimensional unstable manifolds
$\Gamma^u=\{\gamma^u\}$ and a continuous family of $(\dim(M)-n)$-dimensional stable manifolds
$\Gamma^s=\{\gamma^s\}$ such that
\begin{enumerate}
    \item $\Lambda=(\cup\gamma^u)\cap(\cup\gamma^s)$;
    \item each $\gamma^s$ meets each $\gamma^u$ in exactly one point, with the angle of $\gamma^s$ and~$\gamma^u$ uniformly bounded away from zero.
    \end{enumerate}
\end{Definition}

\begin{Definition}
Let $\Lambda\subset M$ have a product structure defined by families $\Gamma^s$ and $\Gamma^u$. A subset $\Lambda_0\subset
\Lambda$ is an {\em $s$-subset} if $\Lambda_0$ has a
hyperbolic product structure defined by families $\Gamma_0^s\subset\Gamma^s$ and $\Gamma_0^u=\Gamma^u$; {\em
$u$-subsets} are defined similarly.
\end{Definition}

For $*=u,s$, given $x\in\Lambda$, let $\gamma^{*}(x)$ denote the element of
$\Gamma^{*}$ containing $x$, and let
$f^*$ denote the restriction of the map $f$ to
$\gamma^*$-disks and $|\det Df^*|$ denote the Jacobian of
$Df^*$.

\begin{Definition}\label{d.product}
 A set $\Lambda$ with a product structure for which properties ($\bf P_0$)-($\bf P_4$) below hold will be called a \emph{Gibbs-Markov-Young (GMY) structure}. From here on we assume that $C>0$, $0<\beta<1$ and $0<\zeta\le 1$ are constants
depending only on~$f$ and~$\Lambda$.
\begin{enumerate}
    \item[\bf\quad(P$\bf_0$)] \emph{Lebesgue detectable}: for every
    $\gamma\in\Gamma^u$, we have
    $\leb_{\gamma}(\Lambda\cap\gamma)>0$;
    \item[\bf\quad
    (P$\bf_1$)] \emph{Markov partition and recurrence times}: there are finitely or countably many pairwise disjoint $s$-subsets $\Lambda_1,\Lambda_2,\dots\subset\Lambda$ such
    that
    \begin{enumerate}
 \item for each $\gamma\in\Gamma^u$, $\leb_{\gamma}\big((\Lambda\setminus\cup\Lambda_i)\cap\gamma\big)=0$;
 \item for each $i\in\NN$ there is integer $R_i\in\NN$ such that $f^{R_i}(\Lambda_i)$ is $u$-subset,
         and for all $x\in \Lambda_i$
        \begin{quote}$
        \displaystyle
         f^{R_i}(\gamma^s(x))\subset \gamma^s(f^{R_i}(x))\qand
         f^{R_i}(\gamma^u(x))\supset \gamma^u(f^{R_i}(x)).
         $
         \end{quote}
         We define the \emph{recurrence time} function $R\colon \cup_{i}\Lambda_{i}\to \NN$ as
         $R\vert_{\Lambda_i}=R_i$. We call $f^{R_i}:\Lambda_i\to\Lambda$ the \emph{induced
         map}.
    \end{enumerate}
\end{enumerate}

 \begin{enumerate}
\item[\bf(P$\bf_2$)] \emph{Uniform contraction on stable leaves}:
   for each $x\in\Lambda$, $y\in\gamma^s(x)$ and $ n\ge
 1$
 \begin{quote}$\displaystyle\dist(f^n(y),f^n(x))\le C\beta^n.$
 \end{quote}

 \end{enumerate}

 \begin{enumerate}
\item[\bf(P$\bf_3$)] \emph{Backward contraction and bounded distortion on unstable leaves}:  for all $x,y\in\Lambda_i$ with $ y\in\gamma^u(x) $, and $ 0\le n<R_i$
\begin{enumerate}
  \item \begin{quote}
$\displaystyle\dist(f^n(y),f^n(x))\le C\beta^{{R_i}-n}
\dist(f^{R_i}(x),f^{R_i}(y));$\end{quote}
  \item \begin{quote}
    $\displaystyle\log\frac{\det D(f^{R_i})^u(x)}{\det D(f^{R_i})^u(y)}\le
    C\dist(f^{R_i}(x),f^{R_i}(y))^{\zeta}.$
    \end{quote}
\end{enumerate}
\end{enumerate}



\begin{enumerate}
\item[\bf(P$\bf_4$)] \emph{Regularity of the foliations}:
\begin{enumerate}
 \item \emph{Convergence of $D(f^i|\gamma^u)$}:
  for all $y\in\gamma^s(x)$ and $ n\ge 0$
 \begin{quote}$\displaystyle
 \log \prod_{i=n}^\infty\frac{\det Df^u(f^i(x))}{\det Df^u(f^i(y))}\le C\beta^{n};
 $\end{quote}
 \item \emph{Absolutely continuity of the stable foliation}:
  given $\gamma,\gamma'\in\Gamma^u$, define the holonomy map
$\phi\colon\gamma\cap\Lambda\to\gamma'\cap\Lambda$  as
$\phi(x)=\gamma^s(x)\cap \gamma$. Then $\phi$ is absolutely
continuous with
        \begin{quote}$
        \displaystyle \frac{d(\phi_* \leb_{\gamma})}{d\leb_{\gamma'}}(x)=
        \prod_{i=0}^\infty\frac{\det Df^u(f^i(x))}{\det
        Df^u(f^i(\phi(x)))}.
        $\end{quote}
\end{enumerate}
\end{enumerate}
The notion of absolute continuity is precisely stated in Section~\ref{sec.regularity}.
Under these conditions  we say that  $F=f^{R}:\Lambda\to\Lambda$ is an \emph{induced GMY map}. 
%
%
\end{Definition}

\subsection{Partially hyperbolic attractors}\label{ss.par}
\label{ss.statement} Here we recall the definition of partially hyperbolic attractors with mostly expanding center-unstable direction and then we state our main theorem, Theorem~\ref{t:Markov towers}. This result extends the polynomial estimates in \cite[Theorem A]{AP3} to the (stretched) exponential case.

Let $f:M\to M$ be a $C^{1+}$ diffeomorphism of a finite dimensional
Riemannian manifold~$M$. We say that $f$ is $C^{1+}$ if $f$ is $C^1$ and $Df$ is H\"older continuous. A set $K\subset M$ is said to be invariant if  $f(K)=K$.
\begin{Definition}
A compact invariant subset $K\subset M$ has a \emph{dominated splitting}, if there exists a continuous $Df$-invariant splitting $T_K M=E^{cs}\oplus E^{cu}$ and $0<\lambda<1$ such that (for some choice of Riemannian metric on~$M$) 
 \begin{equation}\label{domination1}
    \|Df \mid E^{cs}_x\|
\cdot \|Df^{-1} \mid E^{cu}_{f(x)}\| \le\lambda,\quad\text{for all
$x\in K$.}
\end{equation}
We call $E^{cs}$ the {\em center-stable bundle} and $E^{cu}$ the
{\em center-unstable bundle}.
\end{Definition}

\begin{Definition}
A compact invariant set $K\subset M$ is called {\em partially hyperbolic\/}, if it has a
dominated splitting $T_K M=E^{cs}\oplus E^{cu}$ for which $E^{cs}$
is \emph{uniformly contracting} or $E^{cu}$ is \emph{uniformly expanding}, i.e. there is $0<\lambda<1$ such that
(for some choice of a Riemannian metric on $M$)
 $$\|Df \mid E^{cs}_x\|\le \lambda \quad \text{or} \quad
 \|Df^{-1} \mid E^{cu}_{f(x)}\|^{-1}\le \lambda,\quad\text{for all
$x\in K$.}
$$
\end{Definition}
In this work we consider partially hyperbolic sets of the same type
of those considered in~\cite{ABV}, for which the center-stable
direction is uniformly contracting and the central-unstable
direction is non-uniformly expanding. To emphasize that, we shall write $E^s$
instead of $E^{cs}$.

\begin{Definition}Given $b>0$, we say that $f$ is
\emph{non-uniformly expanding} at a point $x\in K$ in the central-unstable direction, if
\begin{equation}\label{Nue}\tag{NUE}
 \limsup_{n\to+\infty} \frac{1}{n}
    \sum_{j=1}^{n} \log \|Df^{-1} \mid E^{cu}_{f^j(x)}\|<-b.
\end{equation}
If $f$ satisfies \eqref{Nue} at some point $x\in K$, then
the \emph{expansion time} function at $x$
\begin{equation}\label{exptime}
    \mathcal E(x) = \min\left\{N\ge 1\colon  \frac{1}{n}
\sum_{i=1}^{n} \log \|Df^{-1}\mid E^{cu}_{f^{i}(x)}\| < -b, \quad
\forall\,n\geq N\right\}
\end{equation}
is defined and finite. We call $\{\mathcal E>n\}$ \emph{the tail of hyperbolic times} (at time~$n$).
\end{Definition}

We remark that if condition \eqref{Nue} holds for every point in a subset with positive Lebesgue measure of a forward invariant set $\tilde K\subset  M$, then $K=\cap_{n\ge 0}f^n(\tilde K)$ contains some local unstable disk $D$ for which condition \eqref{Nue} is satisfied $\leb_D$ almost everywhere; see  \cite[Theorem~A]{AP3}.




\cmt \label{t:Markov towers}
Let \( f: M\to M \) be a \(C^{1+}\) diffeomorphism with $K\subset M$ an invariant transitive partially hyperbolic set. Assume that there
are a local unstable disk $D\subset K$ and constants \(
0<\tau\leq1\), $c>0$ such that $\leb_D\{\mathcal E>n\}=\mathcal
O(e^{-cn^{\tau}}).$ Then there exists $\Lambda\subset~K$ with a GMY structure. Moreover,
 there exists $d>0$ such that
 $\leb_\gamma\{ R> n\} =\mathcal O(e^{-dn^{\tau}})$
 for any $\gamma\in \Gamma^u$.
  \fmt

%

The proof of this result will be given in Section~\ref{se.Gibbs-Markov-Young structure}.
Under the assumptions of Theorem~\ref{t:Markov towers}, the set
$\Lambda$ coincides with $\Gamma^u$, but there are other possibilities, e.g.  in \cite{BY2} $\Lambda$
is a Cantor set for the H\'{e}non attractors.

In Section~\ref{s.appli} we present an open class of diffeomorphisms for which $K=M$ is partially
hyperbolic and satisfies the assumptions of Theorem~\ref{t:Markov
towers}. The transitivity of the diffeomorphisms in that class was proved in~\cite{T}.

\subsection{Statistical properties}
A good way of describing the dynamical behavior of chaotic dynamical
systems 
is through invariant probability measures;
in our setting, a special role is played by SRB measures.

\cd
An $f$-invariant probability measure $\mu$ on the Borel sets of $M$ is called a \emph{Sinai-Ruelle-Bowen (SRB) measure} if $f$ has no zero Lypaunov exponents $\mu$ almost everywhere and the conditional measures of $\mu$ on local unstable manifolds are absolutely continuous with respect to the Lebesgue measure on these manifolds.
\fd

It is well known that SRB measures are \emph{physical measures}:  for a positive Lebesgue
measure set of points $x\in M$,
\begin{equation}
\lim_{n\to+\infty} \frac{1}{n}
     \sum_{j=0}^{n-1} \varphi(f^j(x)) = \int \varphi\,d\mu,
\quad\text{for any continuous } \varphi:M\to\RR.
\end{equation}

SRB measures for partially hyperbolic diffeomorphisms whose central direction is non-uniformly expanding  were already obtained in~\cite{ABV}. Under the assumptions of Theorem~\ref{t:Markov towers}, we also get the existence of such measures by means of \cite[Theorem 1]{Y1}.

\cd Given observables $\varphi,
\psi\colon M\to \RR$, we define the \emph{correlation function} with respect to a measure $\mu$ as
\[
\mathcal C_{\mu}(\varphi, \psi\circ f^n) = \left|\int \varphi (
\psi\circ f^n)\,d\mu - \int \varphi\,d\mu \int \psi\,d\mu\right|, \quad n\ge 0.
\]
\fd

Sometimes it is possible to obtain specific rates for which
$\mathcal C_{\mu}(\varphi, \psi)$ decays to 0 as $n\to\infty$, at least for certain classes of observables with some
regularity. See that if we take the observables
as characteristic functions of Borel sets, we get the classical
definition of \emph{mixing}.

The next corollary follows from
Theorem~\ref{t:Markov towers} together with \cite[Theorem~B]{AP2}; see also \cite[Remark~2.4]{AP2}. Though
in \cite{AP2} the decay of correlations depends on some
backward decay rates in the unstable direction, in our case we
clearly have exponential backward contraction along that direction. So the next result is indeed an extension of~\cite[Corollary~B]{AP3} to the (stretched) exponential case.

\cmc[Decay of Correlations] \label{c:decay} Let \( f: M\to M \) be a \(C^{1+}\) diffeomorphism with an invariant transitive partially hyperbolic set $K\subset M$. Assume that there
are a local unstable disk $D\subset K$ and constants \(
0<\tau\leq1\), $c>0$ such that $\leb_D\{\mathcal E>n\}=\mathcal
O(e^{-cn^{\tau}}).$ Then some power \( f^k \) has an SRB
measure \( \mu \) and there is $d>0$ such that
    $\mathcal C_{\mu}(\varphi,\psi\circ f^{kn}) = \mathcal O(e^{-dn^{\tau}})$
    for H\"older
continuous \(\varphi\colon M\to\RR \), and $\psi\in L^{\infty}(\mu)$.
 \fmc

If the recurrence times associated to the elements of the GMY structure given
by Theorem~\ref{t:Markov towers} are relatively prime, i.e.
gcd$\{R_i\}=1$, then the same conclusion holds with respect to~$f$, i.e. for $k=1$.
Using Theorem~\ref{t:Markov towers}
and \cite[Theorem 4.1]{MN1}, we also deduce a large deviations result for $f$.

\cmc[Large Deviations]\label{c:deviation} Let \( f: M\to M \) be a \(C^{1+}\) diffeomorphism with an invariant transitive partially hyperbolic set $K\subset M$. Assume that there
are a local unstable disk $D\subset K$ and  $c>0$ such that $\leb_D\{\mathcal E>n\}=\mathcal
O(e^{-cn}).$
    Given any H\"{o}lder continuous $\varphi:M\to
\mathbb R$, the limit
$$
\sigma^2= \lim_{n\to\infty}\frac1n\int\left(\sum_{j=0}^{n-1}\varphi\circ f^j-n\int\varphi\,d\mu\right)^2
d\mu
$$
exists. Moreover, if $\sigma^2>0$, then there is a rate function $c(\epsilon)$ such that
$$\lim_{n\to\infty}\frac1n\log\mu\left( \left|\frac{1}{n}\sum_{j=0}^{n-1}\varphi\circ f^j-\int\varphi\,
d\mu\right|>\epsilon \right)=-c(\epsilon).$$
\fmc

We observe that in this last result we do not need to take any power of $f$; see the considerations in \cite[Section~2.2]{MN1}.
It remains an interesting open question to know whether we have a similar result in the stretched exponential case; this depends only on a stretched exponential version of \cite[Theorem 4.1]{MN1}.
Further statistical properties, as the Central Limit Theorem or an Almost Sure Invariant Principle, which have already been obtained in~\cite{AP3}, could still  be deduced from Theorem~\ref{t:Markov towers}.



\section{Preliminary results}\label{se.preliminaries}

 In this section we make a  revision of some concepts and results from \cite{ABV} that will be useful for the proof of Theorem~\ref{t:Markov towers}. In particular,  we state a bounded distortion property at hyperbolic times for iterations of $f$ over center-unstable disks with a H\"older control on the tangent direction.

First we give the  definition of the center-unstable cone field. We consider continuous extensions of $E^{s}$ and $E^{cu}$ to some neighborhood $U$ of $K$ that we denote
by $\tilde{E}^{s}$ and~$\tilde{E}^{cu}$, respectively. These extensions are not necessarily invariant under $Df$. 
\begin{Definition}
Given $0<a<1$, the {\em center-unstable cone field
$C_a^{cu}=\left(C_a^{cu}(x)\right)_{x\in U}$ of width~$a$\/} is defined by
\begin{equation*}
 C_a^{cu}(x)=\big\{v_1+v_2 \in \tilde{E}_x^{s}\oplus
\tilde{E}_x^{cu} \text{\ such\ that\ } \|v_1\| \le a \|v_2\|\big\};
\end{equation*}
the {\em stable cone field
$C_a^{s}=\left(C_a^{s}(x)\right)_{x\in U}$ of width $a$\/} is defined similarly,
\begin{equation*}
 C_a^{s}(x)=\big\{v_1+v_2 \in \tilde{E}_x^{s}\oplus
\tilde{E}_x^{cu} \text{\ such\ that\ } \|v_2\| \le a \|v_1\|\big\}.
\end{equation*}
\end{Definition}

Notice that the dominated splitting property still holds for the extensions $\tilde{E}^{s}$ and $\tilde{E}^{cu}$, provided $U$  is taken sufficiently small. Up to slightly increasing $\lambda<1$, we fix $a>0$ and $U$ small enough so that the domination condition~
\eqref{domination1} still holds for any
point $x\in U\cap f^{-1}(U)$ and every $v^{s}\in C_a^{s}(x)$, $v^{cu}\in C_a^{cu}(f(x))$:
$$
\|Df(x)v^{s}\|\cdot\|Df^{-1}(f(x))v^{cu}\|
\le\lambda\|v^{s}\|\,\|v^{cu}\|.
$$
The center-unstable cone
field is forward invariant
 $$Df(x) C_a^{cu}(x)\subset
C_a^{cu}(f(x)),\quad \mbox{any}\,\, x\in K,$$ and this holds for any $x\in U\cap
f^{-1}(U)$ by continuity.


We say that an embedded $C^1$ submanifold $N\subset U$ is \emph{tangent to the center-unstable cone field} or a \emph{center-unstable (cu) disk} if the tangent subspace to $N$ at each point $x\in N$ is 
contained in the corresponding cone $C^{cu}_a (x)$.

\subsection{H\"older control of the tangent direction}
The goal of this subsection is to introduce  $\kappa(N)$ in \eqref{e.kappa} for a submanifold $N\subset U$ and the constant $C_1$ in
Proposition~\ref{c.curvature}. This will be useful for the statement  of Proposition~\ref{l.contraction}.

  The notion of H\"older variation of the tangent bundle for a centre-unstable manifold $N\subset U$ is introduced
in local coordinates as follows.
First we take $\delta_0>0$ small enough so that the inverse of the
exponential map $\exp_x$ is defined on the $\delta_0$ neighborhood
of every point $x\in U$.
Then we identify this neighborhood of $x$ with the corresponding
neighborhood $V_x$ of the origin in $T_x N$, through the local chart defined by $\exp_x^{-1}$.
Accordingly, we identify $x$ with $0\in T_xN$.
Reducing $\delta_0$, if necessary, we may suppose that
$\tilde{E}^{cs}_x$ is contained in the center-stable cone
$C^{cs}_a(y)$ of every $y\in V_x$.
In particular, the intersection of $C^{cu}_a(y)$ with
$\tilde{E}^{cs}_x$ reduces to the zero vector.
Then, the tangent space to $T_yN$ is parallel to the graph
of a unique linear map $A_x(y):T_x N \to \tilde{E}_x^{cs}$.
Given constants $C>0$ and $0<\zeta\le 1$, we say that
{\em the tangent bundle to $N$ is $(C,\zeta)$-H\"older\/} if
\begin{equation*}
\label{e.holder}
\|A_x(y)\|\le C d_x(y)^\zeta
\quad\text{for every }y\in N \cap V_x \text{ and } x\in U ,
\end{equation*}
where $d_x(y)$ denotes the distance from $x$ to $y$ measured along $N\cap V_x$,
defined as the length of the shortest curve in $N\cap V_x$ joining $x$ to $y$.

Recall that we have taken the neighborhood $U$ and the cone
width $a$ sufficiently small so that the domination property
remains valid for vectors in the cones $C_a^{cs}(z)$,
$C_a^{cu}(z)$, and for any point $z\in U$.
Hence, there are $\lambda_1 \in (\lambda,1)$ and $\zeta\in(0,1]$ such
that
\begin{equation}
\label{e.dominacao}
\|Df(z) v^{cs}\| \cdot \|Df^{-1}(f(z)) v^{cu}\|^{1+\zeta} \le \lambda_1 < 1
\end{equation}
for all unit vectors $v^{cs}\in C_a^{cs}(z)$
and $v^{cu}\in C_a^{cu}(z)$, with $z\in U$.
Then, up to reducing $\delta_0>0$ and slightly increasing $\lambda_1<1$, inequality 
(\ref{e.dominacao}) still holds if we replace $z$ by any $y\in V_x$,
$x\in U$ (where $\|\cdot\|$ means the Riemannian metric in the
corresponding local chart).

From here on we fix  $\lambda_1 \in (\lambda,1)$ and $\zeta\in(0,1]$ as above.
Given a $C^1$ submanifold $N\subset U$, we define
\begin{equation}
\label{e.kappa}
\kappa(N)=\inf\{C>0:\text{the tangent bundle of $N$ is $(C,\zeta)$-H\"older}\}.
\end{equation}
The proof of the next result is given in
\cite[Corollary 2.4]{ABV}.

\cpr
\label{c.curvature}
There exists $C_1>0$ such that for any $C^1$ submanifold  $N\subset U$ tangent to the center-unstable cone field
\begin{enumerate}
\item
there exists $n_0\ge 1$ such that $\kappa(f^n(N)) \le C_1$ for every
$n\ge n_0$ such that $f^k(N) \subset U$ for all $0\le k \le n$;
\item
if $\kappa(N) \le C_1$, then $\kappa(f^n(N)) \le C_1$ for each
$n\ge 1$ such that $f^k(N)\subset U$ for all $0\le k \le n$;
\item
 if $N$ and $n$ are as in the previous item, then the functions
$$
J_k: f^k(N)\ni x \longmapsto \log |\det \big(Df \mid T_x f^k(N)\big)|,
\quad\text{$0\le k \le n$},
$$
are $(L_1,\zeta)$-H\"older continuous with $L_1>0$
depending only on $C_1$ and $f$.
\end{enumerate}
\fpr



\subsection{Hyperbolic times and bounded distortion} \label{sub.hyp}
We can derive uniform expansion and bounded distortion from NUE
assumption in the center-unstable direction, with the definition
below. Here we do not need the full strength of partially hyperbolic, we only consider the cu-direction has condition~\eqref{Nue}.

\cd \label{d.hyperbolic1} Given $0<\sigma<1$, we say that $n$ is a
{\em $\sigma$-hyperbolic time} for  $x\in K$ if
$$
\prod_{j=n-k+1}^{n}\|Df^{-1} \mid E^{cu}_{f^{j}(x)}\| \le \sigma^k,
\qquad\text{for all $1\le k \le n$.}
$$
For $n\ge 1$, we define
 $$
 H_n(\sigma)=\{x\in K\colon \text{ $n$ is a $\sigma$-hyperbolic time for
 $x$ }\}.
 $$
\fd

\cre Given $0<\sigma<1$ and $x\in H_n(\sigma)$, we obtain
\begin{equation}\label{contra}
    \|Df^{-k} \mid E^{cu}_{f^{n}(x)}\| \le \prod_{j=n-k+1}^{n}\|Df^{-1}
\mid E^{cu}_{f^{j}(x)}\| \le \sigma^{k},
\end{equation}
which means that $Df^{-k} \mid E^{cu}_{f^{n}(x)}$ is a contraction for
$1\le k\le n$.
\fre

The next result gives the existence of $\sigma$-hyperbolic times
for almost all points in a center-unstable disk $D$ as in Theorem~\ref{t:Markov towers}, and gives indeed the asymptotic positive frequency of $\sigma$-hyperbolic times for such points. Its proof
can be found in \cite[Lemma 3.1, Corollary~3.2]{ABV}.

\cpr\label{l:hyperbolic2}
    There exist \( 0<\theta\le 1 \) and $0<\sigma<1$ such that for every $x\in D$ with $\mathcal E(x)\le n$  there exist $\sigma$-hyperbolic times $1\le n_1<\dots<n_\ell\le n$ for $x$ with $\ell\ge\theta n$.
\fpr

We remark that both $\theta$ and $\sigma$ are uniform constants independent  of the point $x$ or the iterate $n$.  In the sequel, we fix $0<\sigma<1$ as in the previous proposition and write simply $H_n$ for $H_n(\sigma)$.

\cre \label{re.delta1} By continuity, we may choose $a>0$ (recall the definition of the cone-fields) 
and $\delta_1>0$ sufficiently small (in particular, the
$\delta_1$-neighborhood of $K$ must be contained in $U$) such that
\begin{equation}
\label{e.delta1} \|Df^{-1}(f(y)) v \| \le \frac{1}{\sqrt\sigma}
\|Df^{-1}|E^{cu}_{f(x)}\|\,\|v\|,
\end{equation}
whenever $x\in K$, $\dist(y,x)\le\delta_1$ and $v\in C^{cu}_a(y)$.
\fre

By the first item of Proposition~\ref{c.curvature} we may assume that the center-unstable disk $D\subset K$ in the statement of Theorem~\ref{t:Markov towers} satisfies $\kappa(D)\le C_1$. The next result is then a consequence \cite[Lemma~2.7 \& Proposition~2.8]{ABV}.


\cpr \label{l.contraction} 
There exists $C_2>1$ such that for any 
 $x\in  D\cap H_n$ at a positive distance from $\partial D$, for  $n$ sufficiently large   there is a neighborhood $V_n(x)$  of $x$ in $ D$ such that:
\begin{enumerate}
    \item $f^{n}$ maps $V_n(x)$ diffeomorphically onto a center-unstable disk $B^{cu}(f^{n}(x),\delta_1)$;
    \item for every $1\le k
\le n$ and $y, z\in V_n(x)$
 $$\dist_{f^{n-k}(V_n(x))}(f^{n-k}(y),f^{n-k}(z)) \le
\sigma^{k/2}\dist_{f^n(V_n(x))}(f^{n}(y),f^{n}(z));$$
  \item \label{p.distortion} for all $y,z\in V_n(x)$
$$\log \frac{|\det Df^{n} \mid T_y \Delta|}
                     {|\det Df^{n} \mid T_z \Delta|}
            \le C_2 \dist_{f^n(D)}(f^{n}(y),f^{n}(z))^\zeta;$$
\item for any
Borel sets $Y, Z\subset V_n(x)$ 
$$
\frac{1}{C_2}\frac{\leb(Y)}{\leb(Z)}\leq
\frac{\leb(f^n(Y))}{\leb(f^n(Z))} \leq C_2\frac{\leb(Y)}{\leb(Z)}.
$$
\end{enumerate}
\fpr
The sets \( V_n (x)\) will be called \emph{hyperbolic pre-balls}, and their
images $B^{cu}(f^n(x),\delta_1)$  called \emph{hyperbolic balls}.
Item (\ref{p.distortion}) gives the \emph{bounded distortion} at hyperbolic times.

%
%
%

\section{Partition on a reference disk}
\label{se.Gibbs-Markov-Young structure}

In this section we prove the existence of  a set with product structure in
 $K$. We essentially describe the geometrical and dynamical
nature. This process has three steps. Firstly we prove the
existence of a center-unstable disk $\Delta_0$ whose hyperbolic pre-disks contained in it return to a neighborhood of $\Delta_0$ under forward iterations and the image projects along stable leaves covering $\Delta_0$ completely.
Secondly, we define a partition on $\Delta_0$ whose
construction is inspired essentially on \cite[Section 3]{AP3} and
\cite[Section 3 \& 4]{ADL}. That is, we improve the product structure construction performed in  \cite{ADL} for non-invertible NUE maps, extending it to the partially hyperbolic setting; see Subsections~\ref{s.partition} and~\ref{s.product}. Finally we show that the set with a product structure satisfies Definition \ref{d.product}.

\subsection{The reference disk}
\label{s.disk}

Let $D$ be a local unstable disk as in  Theorem~\ref{t:Markov towers}.  Given $\delta_1>0$ as in Remark~\ref{re.delta1}, we take $0<\delta_s<\delta_1/2$ such that points in
$K$ have local stable manifolds of radius $\delta_s$. In particular,
these local stable leaves are contained in the neighborhood~$U$ of~$K$; recall
\eqref{e.delta1}.

\cd\label{d.ucross}  Given a disk $\Delta\subset D$, we define the \emph{cylinder} over $\Delta$ $$\cc(\Delta)=\bigcup_{x\in\Delta}W^s_{\delta_s}(x)$$
and consider $\pi$ the projection from $\cc(\Delta)$ onto $\Delta$
along local stable leaves.
We say that a center-unstable disk $\gamma^u$ {\em $u$-crosses} $\cc(\Delta)$ if $\pi(\gamma^u\cap \mathcal C(\Delta))=\Delta.$
\fd

For technical reasons  (see Lemma~\ref{l.in V_n}) we shall take the constant $$\delta'_1=\frac{\delta_1}{12}>0,$$ and consider $V'_n(x)$ the part of $V_n(x)$ which is sent by $f^n$ onto $B^{cu}(f^n(x),\delta'_1)$. These sets $V'_n(x)$  will also be called hyperbolic pre-balls.
The next lemma is a consequence of \cite[Lemma 3.1 \& 3.2]{AP3}.

\cle\label{l.N0q}
There are   $p\in D$ and $N_0\ge 1$ such that for all $\delta_0>0$ sufficiently small
and each hyperbolic pre-disk  $V'_n(x)\subseteq D$ there is $0\le m\le N_0$ such that  $f^{n+m}(V'_n(x))$  $u$-crosses $\mathcal C(\Delta_0)$, where $\Delta_0=B^{cu}(p,\delta_{0})$  is the
subdisk in $D$ of radius $\delta_0$  centerd at \( p \).
\fle


Now we fix $p\in D$, $N_0\ge 1$ and $\delta_0>0$ small enough such that the conclusions of Lemma~\ref{l.N0q} hold and define
$$\Delta_0=\Delta_0^0=
B^{cu}(p,\delta_{0})\qand  \Delta_0^1= B^{cu}(p,2\delta_{0}).
$$
We consider the corresponding cylinders
\begin{equation}\label{eq.C0}
 \cc^i_0=\bigcup_{x\in\Delta_0^i
}W^s_{\delta_s}(x), \quad\text{ for }
 i=0,1.
\end{equation}
Denoting $\pi$ the projection along stable leaves, we have
$$\pi(\cc^i_0)=\Delta_0^i,\quad\mbox{for }i=0,1.$$

\cre We assume that each disk $\gamma^u$ $u$-crossing $ \cc^i_0$ ($i=0,1$) is a disk
centered at a point of $W^s_{\delta_s}(p)$ and with the same radius
of $\Delta^{i}_0$. We ignore the difference of radius caused by the height of the cylinder and the angles of the two dominated splitting bundles. Let the top and bottom components of $\partial  \cc^1_0$ be denoted by  $\partial^u  \cc^1_0$, i.e. the set of points $z\in\partial \cc^1_0$ such that $z\in
\partial W^s_{\delta_s}(x)$ for some $x\in \Delta_0$. By the
domination property, we may take $\delta_0>0$ small enough so that any
center-unstable disk $\gamma^u$ which is contained in $ \cc^1_0$ and intersecting
$W^s_{\delta_s/2}(p)$ does not reach $\partial^u \cc^1_0$.\fre

Given a hyperbolic pre-ball $V'_n(x)$, there is $0\le m\le N_0$ as in the conclusion of Lemma~\ref{l.N0q}, and for each $i=0,1$ there is a center-unstable disk 
$\omega_{n,m}^{i,x}\subset V'_n(x)$ such that
\begin{equation}\label{D.candidate2}
 \pi\left(f^{n+m}(\omega_{n,m}^{i,x})\right)=\Delta_0^i.
\end{equation}
As condition \eqref{D.candidate2} may in principle hold for several values of $m$, for definiteness we shall always assume that $m$ takes the smallest possible value. Observe that the center-unstable disk $\omega_{n,m}^{i,x}$ is associated to $x$, by construction, but does not necessarily contain~$x$.

The sets of the type $\omega_{n,m}^{0,x}$, with $x\in H_n\cap\Delta_0$, are the natural candidates to be in the partition~$\mathcal P$.
For $k\ge n$, set the \emph{annulus} around $\omega_{n,m}^{0,x}$
\begin{equation}\label{eq.Annulus}
A_{k}(\omega_{n,m}^{0,x})=\left\{y\in \omega_{n,m}^{1,x}:0<\dist_D\left((\pi\circ f^{n+m})(y),\Delta_0\right)\le \delta_0\sigma^{\frac{k-n}{2}}\right\}.
\end{equation}
Obviously $$A_{n}(\omega_{n,m}^{0,x})\cup\omega_{n,m}^{0,x}=\omega_{n,m}^{1,x}.$$
In the sequel, we shall frequently omit the symbols $m$, $0$, $x$ or $n$ in the notation
 and simply use $\omega_n^x$, $\omega_n$ or even $\omega$ to denote an element  $\omega_{n,m}^{0,x}$.

\subsection{The partition}
\label{s.partition}

In this subsection we describe an algorithm to construct  a countable ($\leb_D$ mod 0)  partition $\cp$ of $\Delta_0$. The algorithm is similar
to the one in \cite{ADL}, but in the present context of a diffeomorphism, each element of the partition will
return to another center-unstable disk which \emph{u-crosses} $\mathcal C_0^0$.
Along the process we shall introduce inductively sequences of objects
$(\Delta_n)_n$, $(\Omega_n)_n$, $(A_n)_n$ 
and $(S_n)_n$. For each $n$, 
 $\Delta_n$ is the set of points which  does not belong to any element of the partition constructed up to time $n$,
$\Omega_n$ is the union of elements of the partition constructed at step $n$ and $A_n$ is the union of rings around the chosen elements at time $n$. 
 The set
$S_n$  (called the union of \emph{satellites}) contains the components which could have been chosen for the partition but intersect already chosen elements. 
A key point in our argument is property \eqref{Hdel_n} below, which says that every point having a hyperbolic time at a given time $n$ will belong to either to an element of the partition or to some satellite. All these  and some other auxiliary objects will be defined inductively in the remaining part of this subsection.

\subsubsection*{First step of induction}

Fixing some large \(n_0\in\mathbb{N} \), we only consider the dynamics after time
$n_0$.  Notice that, by boundedness on the derivative, there is a minimum radius $r_{n_0}>0$  such that each hyperbolic pre-disk $V'_{n_0}(x)$  with $x\in H_{n_0}$ contains a center-unstable disk of radius $r_{n_0}$. Hence, there is a finite set 
$I_{n_0}=\{z_1,\ldots,z_{N_{n_0}}\}\in H_{n_0}\cap\Delta_0$ such that
$$H_{n_0}\cap \Delta_0\subset V'_{n_0}(z_1)\cup\dots\cup
V'_{n_0}(z_{N_{n_0}}).$$
Consider a maximal family 
$$
\Omega_{n_0}=\left\{\omega_{n_0,m_0}^{0,x_0},
\omega_{n_0,m_1}^{0,x_1},\ldots,
\omega_{n_0,m_{k_{n_0}}}^{0,x_{k_{n_0}}}\right\}.
$$
of pairwise disjoint sets of the type (\ref{D.candidate2}) contained in
$\Delta_0$ with $\{x_0,\cdots,x_{k_{n_0}}\}\subset I_{n_0}$, and let
$$ \tilde I_{n_0}= I_{n_0}\setminus \{x_0,\cdots,x_{k_{n_0}}\}.$$
The sets in $\Omega_{n_0}$ are precisely the elements of the partition $\cp$ constructed in the
$n_0$-step  of the algorithm (our first step of induction).
We define the \emph{recurrence
time} $R(x)=n_0+m_i$ for each
$x\in\omega^{0,x_i}_{n_0,m_i}$ with $0\leq i \leq k_{n_0}$.
We need to keep track of the sets $\{\omega^{1,z}_{n_0,m}: z\in
\tilde I_{n_0},0\le m\le N_0\}$ which, for some $\omega\in \Omega_{n_0}$, overlap $\omega\cup A_{n_0}(\omega)$ or $\Delta_0^c=D\setminus\Delta_0$.
%
Given $\omega\in \Omega_{n_0}$, 
we define for each $0\le m\le N_0$
\begin{equation*}
    I_{n_0}^m(\omega)=\left\{x \in \tilde I_{n_0}: \omega^{1,x}_{n_0,m}\cap\left( \omega \cup A_{n_0}(\omega)\right)\neq\emptyset\right\},
\end{equation*}
(recall \eqref{eq.Annulus})
and  the \emph{$n_0$-satellite} around $\omega$
\begin{equation}\label{eq.Sn0}
    S_{n_0}(\omega)=\bigcup_{m=0}^{N_0}\bigcup_{x\in
    I_{n_0}^m(\omega)}V'_{n_0}(x)\cap (\Delta_0\setminus
    \omega).
\end{equation}
We define 
\begin{equation*}
S_{n_0}(\Omega_{n_0}) =\bigcup_{\omega\in
\Omega_{n_0}}{S}_{n_0}(\omega)
\end{equation*}
and $$S_{n_0}(\Delta_0) = S_{n_0}(\Omega_{n_0}).$$
We also define the $n_0$-satellite associated to $\Delta_0^c=D\setminus\Delta_0$
\begin{equation*}
    S_{n_0}(\Delta_0^c)= \bigcup_{m=0}^{N_0}\bigcup_{x\in \tilde I_{n_0}\atop \omega^{1,x}_{n_0,m}\cap\Delta_0^c\neq\emptyset}V'_{n_0}(x)\cap \Delta_{0}, .
\end{equation*}
We will show in the general step that the Lebesgue measure of
$S_{n_0}(\Delta_{0}^c)$ is exponentially small.
The \emph{global $n_0$-satellite} is
\begin{equation*}
{S}_{n_0} = S_{n_0}(\Delta_0) \cup {S}_{n_0}(\Delta_0^c),
\end{equation*}
The remaining points at step $n_0$  are
\begin{equation*} \Delta_{n_0} = \Delta_0\setminus
\cup_{\omega\in \Omega_{n_0}}\omega.
\end{equation*}
We clearly have for this first step of induction the key property 
\begin{equation*}
H_{n_0}\cap\Delta_0\,\subset\, S_{n_0}\cup \bigcup_{\omega\in \Omega_{n_0}}\omega.
\end{equation*}

\subsubsection*{General step of induction}

The general step of the construction follows the ideas of the first step with
minor modifications. Given $n>n_0$, assume that the sets
$\Omega_\ell$, $\Delta_{\ell}$, $A_\ell$ and ${S}_{\ell}$ are defined for each \( n_0\le \ell\leq n-1 \).
As in the first step, there is a finite set of points
$I_{n}=\{z_1,\ldots,z_{N_{n}}\}\in  H_n\cap\Delta_{n-1}$ such that
$$ H_n\cap\Delta_{n-1}\subset
V'_{n}(z_1)\cup\dots\cup V'_{n}(z_{N_{n}}).$$ Now we consider a maximal family 
$$\Omega_n=\{\omega_{n,m_0}^{0,x_0},
\omega_{n,m_1}^{0,x_1},\ldots,
\omega_{n,m_{k_{n}}}^{0,x_{k_{n}}}\}$$
of pairwise disjoint sets of type \eqref{D.candidate2} with $\{x_0,\cdots,x_{k_{n}}\}\subset I_{n}$ contained in
$\Delta_{n-1}$ and
satisfying
$$ \omega_{n,m}^{1,x_i}\cap\left(
\cup_{\ell=n_0}^{n-1}\cup_{\omega\in \Omega_\ell}\left(\omega\cup A_n(\omega)\right)\right)=\emptyset,\quad
\text{ for }i=1,\dots,k_n.$$
The sets in $\Omega_n$ are the elements of the partition $\cp$ constructed in the
$n$-step  of the algorithm.
We set the \emph{recurrence time} $R(x)=n+m_i$ for each $x\in
\omega^{0,x_i}_{n,m_i}$ with $0\leq i \leq \ell_n$. 
Let $$ \tilde I_{n}= I_{n}\setminus \{x_0,\cdots,x_{k_{n}}\}.$$
Given $\omega\in \Omega_{n_0}\cup\cdots\cup \Omega_n$ and $ 0\leq
m\leq N_0 $, set
\begin{equation*}
    I_{n}^m(\omega)=\left\{x \in\tilde I_{n}: \omega^{1,x}_{n,m}\cap
    \left(\omega\cup A_n(\omega)\right)
\neq\emptyset\right\}.
\end{equation*}
(recall \eqref{eq.Annulus}) and the \emph{$n$-satellite} around $\omega$
\begin{equation}\label{eq.Sn}
S_n(\omega)= \bigcup_{m=0}^{N_0}\bigcup_{x\in
    I_{n}^m(\omega)}V'_{n}(x)\cap (\Delta_0\setminus \omega).
\end{equation}
We define for $n_0\le i\le n$
\begin{equation*}
S_{n}(\Omega_{i}) =\bigcup_{\omega\in
\Omega_{i}}{S}_{n}(\omega)
\end{equation*}
and 
\begin{equation*}
    S_{n}(\Delta_0)= \bigcup_{i=n_0}^{n}S_{n}(\Omega_{i}) .
\end{equation*}
Similarly, the \emph{$n$-satellite} associated to $\Delta_0^c$ is
\begin{equation*}
    S_{n}(\Delta_0^c)= \bigcup_{m=0}^{N_0}\bigcup_{x\in \tilde I_{n}\atop\omega^{1,x}_{n,m}\cap\Delta_0^c\neq\emptyset}V'_{n}(x)\cap\Delta_0.
\end{equation*}

\cre\label{lastremark}
Observe that the volume of $S_n(\Delta_0^c)$
decays exponentially fast. In fact, it follows from the definition of
${S}_{n}({\Delta_0^c})$ and Proposition~\ref{l.contraction} that
$$
{S}_{n}({\Delta_0^c})\subset \{x \in \Delta_0:\,
\dist_D(x,\partial\Delta_0)\le 2\delta_0\sigma^{n/2}\}.
$$
Thus, there exists $\rho>0$ such that
$
\leb_D ({S}_{n}({\Delta_0^c}))\leq \rho\sigma^{n/2}.
$
\fre

Finally we define the \emph{global $n$-satellite} 
\begin{equation*}
{S}_{n} = S_n(\Delta_0) \cup {S}_{n}(\Delta_0^c),
\end{equation*}
and
\begin{equation*}
\Delta_{n} = \Delta_0\setminus \bigcup_{i=n_0}^{n}\bigcup_{\omega\in\Omega_i}\omega.
\end{equation*}
We clearly have by construction 
\begin{equation}\label{Hdel_n}  H_n\cap\Delta_0\subset
{S}_{n}\cup\bigcup_{i=n_0}^{n}\bigcup_{\omega\in\Omega_i}\omega.
\end{equation}

\subsection {Estimates on the satellites}
%

For the sake of notational simplicity, we shall avoid the superscript  0 in  the sets $\omega^{0,x}_{n,m}$.
The next lemma shows
that, given $n$ and $m$, the conditional volume of the union of
$\omega^{x}_{n,m}$ which is not far from one chosen element is
proportional to the conditional volume of this element. The
proportion constant is uniformly summable with respect to $n$.

Though we consider here the case of partially hyperbolic attractor
and  the construction has naturally obvious modifications, the proofs of the two items in the next
lemma are  essentially the same of \cite[Lemmas 4.4 \&
4.5] {ADL}.

\cle \label{l.preSn}
\begin{enumerate}
\item \label{estpreball} There
exists $C_{3}>0$ such that, for any $n\ge n_0$, $0\leq m\leq N_0$,
and finitely many points $\{x_1,\dots ,x_N\}\in I_n$ satisfying
$\omega_{n,m}^{x_i}=\omega_{n,m}^{x_1}$ ($1\le i\le N$), we have
\begin{equation*}
    \leb_D\left(\bigcup_{i=1}^{N} V'_n(x_i)\right)\leq C_3\leb_D(
    \omega^{x_1}_{n,m}).
   \end{equation*}
\item \label{d.estimativas}There exists $C_4> 0$ such that for $k\ge n_0$, $\omega\in \Omega_k$ and $0\le m\le N_0$, given any $n\ge k$, we obtain
\begin{equation*}
\leb_D \left(\bigcup_{x \in I_n^m(\omega)}\omega_{n,m}^x\right)\leq
C_4\sigma^{\frac{n-k}{2}} \leb_D(\omega).
\end{equation*}

\end{enumerate}
\fle

 \cpr \label{d.prop.Sn}
There exists $C_5>0$ such that for any $\omega\in\Omega_k$ and $n\geq
k$, we have
\begin{equation*}
\leb_D({S}_{n}(\omega))\le C_5\sigma^{\frac{n-k}{2}}\leb_D(\omega).
\end{equation*} \fpr

\dem Consider now $k\ge n_0$ and $n\ge k$. Fix  $\omega \in
\Omega_k$ and consider ${S}_{n}(\omega)$ the $n$-satellite
associated to it. By definition of $S_{n}(\omega)$ and
 the first item of Lemma \ref{l.preSn} we have
\begin{eqnarray*}
  \leb_D(S_{n}(\omega)) &\leq& \sum_{m=0}^{N_0}\sum _{x\in
    I_n^m(\omega)}\leb_D \left( V'_{n}(x)\cap (\Delta_0\setminus \omega)\right)+\leb_D(V'_k(\omega)\setminus\omega)  \nonumber\\
  &\leq& C_3\sum_{m=0}^{N_0}\leb_D \left(\bigcup_{x\in
    I_n^m(\omega)} \omega_{n,m}^x\right) + C_3\leb_D(\omega).
\end{eqnarray*}
In this last step we have used the obvious fact that for fixed $n,m$
the sets of the form $\omega^{x}_{n,m}$ with $x \in I_n^m(\omega)$
are pairwise disjoint.
Thus, by the second item of Lemma~\ref{l.preSn},
\begin{equation*}
  \leb_D(S_{n}(\omega)) \le
  C_3(C_4(N_0+1)+1)\sigma^{\frac{n-k}{2}}\leb_D(\omega).
\end{equation*}
Take $C_5=C_3(C_4(N_0+1)+1)$. \cqd


\cd
Given $k\ge n_0$ and $\omega^x_{k,m}\in \Omega_k$ we define for $n\ge k$
 $$B_n^k(x)=S_n(\omega_{k,m}^{x})\cup
\omega_{k,m}^{x}\quad\text{and}\quad t(B_n^k(x))=k.$$  
The
set $\omega^x_{k,m}$  will be called  the \emph{core} of $B_n^k(x)$
and denoted  $C(B_n^k(x))$.
\fd

Notice that given a set $B^k_n(x)$ as above, 
$k$ is a hyperbolic time for $x$ and $n$ is a hyperbolic time for some point in the $n$-satellite involved in the definition of $B^k_n(x)$; recall the definition of the satellites.

\cre\label{re.cores} From  the construction of the partition $\mathcal P$ performed in Section~\ref{s.partition}, we easily deduce that, given any two cores $C(B_{n_1}^{k_1}(x_1))$ and $C(B_{n_2}^{k_2}(x_2))$, we have either  $x_1=x_2$ and $C(B_{n_1}^{k_1}(x_1))=C(B_{n_2}^{k_2}(x_2))$, or $C(B_{n_1}^{k_1}(x_1))\cap C(B_{n_2}^{k_2}(x_2))=\emptyset$.
\fre

In the sequel we prove  the main features of these sets $B_n^k(x)$. The next result follows immediately from Proposition~\ref{d.prop.Sn}.

\begin{Corollary}\label{l.B}
 For all  $n\ge k$ we have
 $$\leb_D(B_n^k(x))\le
(C_5+1)\leb_D(C(B_n^k(x))).$$
\end{Corollary}

The dependence of $\delta_1'$ on $\delta_1$ becomes clear in the proof of the next lemma.

\begin{Lemma}\label{l.in V_n}
If $\delta_1'>0$ is sufficiently small (only depending on $\delta_1$), then for all $k'\ge k\ge n_0 $, $n\ge k$, $n'\ge k'$ and $B_n^k(x)\cap
B_{n'}^{k'}(y)\neq\emptyset$ we have $$C(B_n^k(x))\cup
C(B_{n'}^{k'}(y))\subset V_k(x).$$
\end{Lemma}
\begin{proof} 
First of all observe that  $k$ is a hyperbolic time for  $x$ and
 $f^k(C(B_n^k(x)))$ is contained in  a $cu$-disk of radius $\delta_1'$ centered at $f^k(x)$.
 On the other hand, by definition  each  point in $B_n^k(x)$ which does not belong to $C(B_n^k(x))$ must necessarily be in some hyperbolic pre-disk (with hyperbolic time $n$) intersecting $C(B_n^k(x))$. Then,  the second assertion of Proposition~\ref{l.contraction} yields 
$$\diam_{f^k(D)}\left(f^k(B_n^k(x))\right)\leq 2\delta'_1+4\delta'_1\sigma^{\frac{n-k}{2}}\le 6\delta_1'.$$
Similarly,
$$\diam_{f^{k'}(D)}\left(f^{k'}(B_{n'}^{k'}(y))\right)\leq  6\delta_1'.$$
Using this and the second assertion of Proposition~\ref{l.contraction}, we also have
$$\diam_{f^{k}(D)}\left(f^{k}(B_{n'}^{k'}(y))\right)\leq  6\delta_1'\sigma^{\frac{k'-k}{2}}\le 6\delta_1'.$$
Hence, as we are taking  $\delta_1'=\delta_1/12$, we have  $f^k(B_n^k(x))\cup f^{k}(B_{n'}^{k'}(y))$  contained in the center-unstable disk of radius $\delta_1$ centered at $f^k(x)$. This clearly gives the result.
%
%
%
\end{proof}

Notice that the proof of the previous lemma gives in fact $B_n^k(x)\cup B_{n'}^{k'}(y)\subset V_k(x)$, but the conclusion with the cores is enough for our purposes. The  sets $A_n$ play a key role in the proof of the next result and have been introduced exclusively to make this proof work.

\begin{Lemma}\label{l.P}
There exists $P\ge N_0$ such that for all $t_2> t_1\ge n_0$, $x\in H_{t_1}$ and $y\in H_{t_2}\setminus B_{t_2}^{t_1}(x)$ we have
$$B_{t_2+P}^{t_1}(x)\cap B_{t_2+P}^{t_2}(y)=\emptyset.$$
\end{Lemma}
\begin{proof}
Notice that by definition we have $t_1$ a hyperbolic time for $x$ and $t_2$ a hyperbolic time for $y$. 
Suppose, by contradiction, that we have
$B_{t_2+P}^{t_2}(y)\cap B_{t_2+P}^{t_1}(x)\neq\emptyset$ for all $P\ge N_0$. Take a
point $z\in B_{t_2+P}^{t_2}(y)\cap B_{t_2+P}^{t_1}(x)$.  Observing that $t_2+P$ is a hyperbolic
time associated to hyperbolic pre-disks intersecting  $C(B_{t_2+P}^{t_1}(x))$ and
$C(B_{t_2+P}^{t_2}(y))$, 
 by the second assertion of
Proposition~\ref{l.contraction} we obtain for $R_1=R(C(B^{t_1}_{t_2+P}(x)))$
\begin{equation*}
    \dist_{f^{R_1}(D)}\left(f^{R_1}(z), f^{R_1}(C(B^{t_1}_{t_2+P}(x)))\right)\leq
    2\delta'_1\sigma^{\frac{t_2+P-R_1}{2}};
\end{equation*}
and also
\begin{equation*}
    \dist_{f^{R_1}(D)}\left(f^{R_1}(z),f^{R_1}(C(B^{t_2}_{t_2+P}(y)))\right)\leq 2\delta'_1\sigma^{\frac{t_2+P-R_1}{2}}.
\end{equation*}
Hence,\begin{equation*}
        \dist_{f^{R_1}(D)}\left(f^{R_1}(C(B^{t_1}_{t_2+P}(x))),f^{R_1}(C(B^{t_2}_{t_2+P}(y)))\right)
        \leq 4\delta'_1\sigma^{\frac{t_2+P-R_1}{2}}.
      \end{equation*}
Recall that, by definition,  $R_1=t_1+m_1$ for some $0\le m_1\le N_0$, and so
$$\frac{t_2+P-R_1}{2}\ge \frac{t_2-t_1+P-N_0}{2}.
$$
Thus, taking $P$ large enough such that
$4\delta'_1\sigma^{P/2}<\delta_0\sigma^{N_0/2}$, we have
$$
\dist_{f^{R_1}(D)}\left(f^{R_1}(C(B^{t_1}_{t_2+P}(x))),f^{R_1}(C(B^{t_2}_{t_2+P}(y)))\right)\leq
\delta_0\sigma^{\frac{t_2-t_1}{2}}
$$
which means $C(B^{t_2}_{t_2+P}(y))\subset
A_{t_2}(C(B^{t_1}_{t_1+P}(x)))$. This gives a contradiction.
\end{proof}

 \section{Tail of recurrence times}
\label{s.tail esti}

Though our construction of the objects  in the previous section  is  significantly  different from \cite{G}, our approach on the estimates below is inspired in  \cite[Section 3.2]{G}.
Our goal in this section is to prove that if the Lebesgue measure of $\{\mathcal E>n\}$ decays (stretched) exponentially fast, then the tail of the recurrence times also decays (stretched) exponentially fast. 
More precisely, we shall see that given a local unstable disk $D\subset K$ and constants $c>0$ and $0<\tau\le 1$, there is $d>0$ such that
\begin{equation}\label{eq.EimpR}
\leb_D\{\mathcal E> n\} =\mathcal O(e^{-cn^{\tau}})\quad\Rightarrow\quad
\leb_D\{R>n\}\leq\mathcal {O}(e^{-dn^{\tau}}).
\end{equation}
First of all
 observe that if there exists $d,\tau>0$ such that
\begin{equation}\label{eq.deltad}
\leb_D(\Delta_n)\leq\mathcal
{O}(e^{-dn^{\tau}}),
\end{equation}
then we have $\mathcal
\{R>n\}\subset\Delta_{n-N_0}$, and so
$$\leb_D\{R>n\}\leq\leb_D(\Delta_{n-N_0})=\mathcal
{O}(e^{-d{(n-N_0)}^{\tau}})=\mathcal {O}(e^{-dn^{\tau}}).$$
Hence, for proving  \eqref{eq.EimpR}, it is enough to see that \eqref{eq.deltad} holds.

By Remark \ref{lastremark}  there exists a constant $\rho >0$ such
that for all $n\in\mathbb N$
\begin{equation}\label{bdy2}
\leb_D\{x\in D \mid
\dist_D(x,\partial\Delta_0)\leq 2\delta_0\sigma^{{\theta n}/{4}}\}\leq
\rho\sigma^{{\theta n}/{4}},
\end{equation}
 where $\theta$ is given in
Proposition~\ref{l:hyperbolic2}.
Take $x\in\Delta_n$ and assume it belongs neither to $\{\mathcal
E>n\}$ nor to
$\{x\,|\,\dist_D(x,\partial\Delta_0)\leq 2\delta_0\sigma^{{\theta n}/{4}}\}$.
Since $n\ge \mathcal E(x)$, by Proposition \ref{l:hyperbolic2} the point $x$ has at least $[\theta n]$
hyperbolic times between $1$ and $n$, and so it has at least
$[{{\theta}n}/{2}]$ hyperbolic times between ${{\theta}n}/{2}$ and~$n$. Ordering  them as ${\theta n}/{2} \le t_1<\dots< t_k\le n$, then $x\in H_{t_i}\cap\Delta_0$ for $1\le i\le k$.
From   \eqref{Hdel_n} we know that
\begin{equation*}
 H_{t_i}\cap\Delta_0\subset {S}_{t_i}\cup
\bigcup_{j=n_0}^{t_i}\bigcup_{\omega\in\Omega_j}\omega, \quad \text{for } 1\le i\le k.
\end{equation*}
If $x\notin S_{t_i}$, then $x\in\cup_{j=n_0}^{t_i}\Omega_j$ and this means that
$x\notin\Delta_n$, which gives a contradiction. Hence $x\in S_{t_i}$.
As $x\in\{x\in\Delta_0\,|\,\dist_D(x,\partial\Delta_0)> 2\delta_0\sigma^{{{\theta}n}/{4}}\}$, we
have
\begin{equation*}
x\in H_{t_i}\cap \{x\in\Delta_0\,\mid\,
\dist_D(x,\partial\Delta_0)> 2\delta_0\sigma^{t_i/2} \},\quad \text{for } 1\le i\le k.
\end{equation*}
Recalling Remark \ref{lastremark}, we obtain
$x\notin S_{t_i}(\Delta_0^c)$, and so $$x\in S_{t_i}(\Delta_0),\quad \text{for } i=1,\dots,k.$$ 
Thus, we have seen that $x$ belongs to the set
$X\left([{\theta}n/{2}],n\right)$,
where we define
\begin{equation*}
X(k,N)= \left\{x\,\mid\,\exists t_1<\ldots<t_k\leq n\textrm{ such that  }x\in
\bigcap_{i=1}^k{S_{t_i}(\Delta_0)}\right\} 
\end{equation*}
for integers $k$  and $N$. 
Hence
$$\Delta_n \subset \{x\in \Delta_0 \mid \mathcal E(x)>n\}\cup \left\{x\in\Delta_0\mid \dist_D(x,\partial\Delta_0)
\leq 2\delta_0\sigma^{\frac{{\theta}n}{4}}\right\}\cup X({[\theta n}/2],n).$$
Since the middle set in the union above has exponentially small measure by \eqref{bdy2}, it remains to see that the measure of
$X([{\theta n}/2],n)$ decays exponentially fast in $n$. 
This follows from the next result, whose proof will be given in the remaining of this section.

\begin{Proposition}\label{X}
There exist
 $D_0>0$ and $0<\lambda_0<1$ such that for all $k$ and $N$
 $$\leb_D(X(k,N))\leq D_0\lambda_0^k\leb_D(\Delta_0).$$
\end{Proposition}

We start by fixing some integer $P'\ge P$ (recall 
Lemma~\ref{l.P}) whose value will be made precise later.
 Given $x\in X(k,N)$, consider all the instants $u_1,\dots,u_p$ for which $x$
belongs to some $S_{u_i+n_i}(\omega_{u_i}^{x_i})$ with $n_i\geq P'$,
ordered so that $u_1<\ldots<u_p$. Defining for $B_1=B^{u_1}_{u_1}(x_1)$
\begin{equation}\label{eq.Yi}
Y(n_1,\ldots,n_p,B_1) = \left\{x\,\mid \,\exists\,t_1<\ldots<t_p   \text{ and $x_2,\ldots,x_p$
 s.t. $x\in
\bigcap_{i=1}^{p}S_{t_i+n_i}\!\!\left(\omega_{t_i}^{x_i}\right)$}\right\},
\end{equation}
we then have $x\in
Y(n_1,\ldots,n_p,B_1)$. 
Assume that $\sum_{i=1}^{p}{n_i}< k/2$. As we are considering $n_1,\ldots,n_k\ge P'$, we must have $p<k/(2P')$. Let $v_1<\ldots<v_q$ be the other instants for which
$x\in S_{v_i+m_i}(\omega_{v_i}^{z_i})$, for times
$m_1,\ldots,m_q<P'$. As  $p+q= k$, for $P'>1$ we have
$$q\geq
\frac{(2P'-1)k}{2P'}\geq\frac{k}{2P'}\ge \left[\frac{k}{2P'}\right].$$  
Thus, considering 
$$
Z(q,N) = \bigg\{ x\in\Delta_N\mid \exists t_1<\ldots<t_q\leq N\text{ and $m_1,\dots,m_q<P'$ s.t. }
 x\in
\bigcap_{i=1}^qS_{t_i+m_i}(\Omega_{t_i})  \bigg\},
$$
we have shown that assuming $\sum_{i=1}^{p}{n_i}< k/2$ we necessarily have $x\in Z\left(\left[{k}/{(2P')}\right],N\right)$. Hence,
\begin{equation}\label{Z12}
X(k,N)\,\,\subset\,\,
\bigcup_{B_1}\bigcup_{n_1,\ldots,n_p\geq
P',\atop\sum{n_i}\geq\frac{k}{2}}Y(n_1,\ldots,n_p,B_1)\cup{Z\left(\left[\frac{k}{2P'}\right],N\right)},
\end{equation}
where the first union is taken over all possible sets $B_1=B^{t_1}_{n_1}(x_1)$.

Our goal now is to obtain estimates on the  $\leb_D$ measure of the sets involved in~\eqref{Z12}. We start with  $Z(k,N)$.
For the sake of notational simplicity, in  the sequel we shall denote for $i\ge0$
$$B_i=B_{t_i+m_i}^{t_i}(x_i)\qand B_i'=B_{t_i'+m_i'}^{t_i'}(x_i').$$ 
We introduce below auxiliary sets $Z_1(k,B_0)$ and $Z_2(k,N)$ which will be useful to estimate the $\leb_D$ measure of $Z(k,N)$.

Given $E\in \NN$, define for a positive integer $k$
 \begin{align*}
 Z_{1}(k,B_{0}) = \bigg\{x\,|\,
\exists B'_{1},B_{1},\ldots, B'_{r},B_{r}\, &\textrm{ so that $B_{i} \nsubseteq B'_{i}$ and   $ t_{i-1} \leq
t'_{i} \leq t_{i}-E$,    $\forall 1
\leq i \leq r$, }\\
 &
 \sum_{i=1}^{r}\left[\frac{t_{i}-t'_{i}}{E}\right] \geq k
\text{ and } \, x \in \bigcap_{i=0}^{r}B_{i}\cap
\bigcap_{i=1}^{r}B'_{i}\,\bigg\}.
\end{align*}
\begin{Lemma}\label{Z_1}
There is $D_1>0$ (independent of $E$) such that for all $k$ and $B_0$
$$
\leb_D(Z_{1}(k,B_0))\leq D_1(D_1 \sigma^{E/2})^{k}\leb_D(C(B_{0})).
$$
\end{Lemma}

\begin{proof}
We shall prove the result by induction on $k\ge 0$.
For $k=0$, we have by Corollary~\ref{l.B} 
\begin{equation*}\label{eq.z1}
 \leb_D(Z_{1}(0,B_0))\le \leb_D(B_0)\leq (C_5+1)\leb_D(C(B_0)),
\end{equation*}
In this case it is enough to take
\begin{equation}
D_1\ge C_5+1
\end{equation}
For $k\geq 1$, we have
\begin{eqnarray*}
Z_{1}(k,B_0)\subset \bigcup_{t=1}^{k}\bigcup_{B'_1\cap B_0\neq
\emptyset}\bigcup_{B'_1\cap B_1\neq \emptyset,B_1\nsubseteq
B'_1,\atop \left[\frac{t_{1}-t'_{1}}{E}\right] \geq t}
Z_1(k-t,B_1).\end{eqnarray*}
Let $n=t_1-t'_1$. Fix some $B'_1$, and take one from all the possible $B_1$'s intersecting $B_1'$. Setting $p=t'_{1}$ and $Q'_{1}=f^p(B'_{1})$, we have
\begin{eqnarray}\label{eq.Z_1}
f^p(B_1)\subset \mathcal
C:=\left\{y\,|\,\dist_{f^p(D)}(y,\partial{Q'_{1}})\leq
6{\delta'_1}{\sigma^{{n}/{2}}}\right\}.
\end{eqnarray}
Indeed, as $B_1$ contains a point of $\partial{B'_1}$, then $f^p(B_{1})$
contains a point of $\partial{Q'_1}$. We obtain
$$\diam_{f^p(D)}{f^p(B_{1})}\leq
{{\sigma^{{n}/{2}}}{\diam_{f^{p+n}(D)}{f^{p+n}(B_{1})}}}\leq
{6{\delta'_1}{\sigma^{{n}/{2}}}},$$ and so  we get \eqref{eq.Z_1}.
Similarly to \eqref{bdy2}, we have for some uniform constant $\rho_1>0$
\begin{equation}\label{eq.medida}
\leb_{f^p(D)}(\mathcal C)\leq \rho_1
\sigma^{n/2}\leb_{f^p(D)}(Q'_1).
\end{equation}
By Remark~\ref{re.cores} the cores $C(B_1)$ of all those possible $B_1$'s  are pairwise disjoint.  Moreover, by Lemma~\ref{l.in V_n} these cores $C(B_1)$ must
be all contained in $V_p(x'_1)$, where $x_1'$ is the point such that $C(B'_1)=\omega^{x'_1}_{t'_1}$.  
Then, using \eqref{eq.Z_1}, \eqref{eq.medida}, Corollary~\ref{l.B} and the bounded distortion  we obtain
\begin{equation}\label{item1}
\sum_{B'_1\cap B_1\neq \emptyset,B_1\nsubseteq
B'_1,\atop \left[\frac{t_1-t'_1}{E}\right]
\geq t}\leb_D(C(B_1)) \leq
{C_2}(C_5+1)\rho_1\sigma^{{Et}/{2}}\leb_D(C(B'_1)).
\end{equation}

Let now $q=t_0$ and $C(B_0)=\omega_{q,m}^{x}$. Once more by Remark~\ref{re.cores} and Lemma~\ref{l.in V_n}, the
possible sets $C(B'_1)$'s are pairwise disjoint and are all
included in $V_q(x)$. 
Moreover, as $f^q(V_q(x))=B^{cu}(f^q(x), \delta_1)$ and $f^{q+m}(C(B_0))$ is a $cu$-disk of radius $\delta_0>0$, there is some uniform constant $\rho_2>0$ such that
$$\leb_{f^q(D)}(B^{cu}(f^q(x),\delta_1))\leq
\rho_2\leb_{{f^q(D)}}(f^q(C(B_0)));$$ 
recall that $0\le m\le N_0$.
Using bounded distortion we get
\begin{equation}\label{eq.bd} \sum_{B'_1\cap B_0\neq
\emptyset}\leb_D(C(B'_1))\leq \leb_D(V_q(x)) \le \rho_2 C_2\leb_D(C(B_0)).
\end{equation}
Finally, using \eqref{item1}, \eqref{eq.bd} and the inductive hypothesis, we deduce
\begin{eqnarray*}
\leb_D(Z_{1}(k,B_0))&\leq& \sum_{t=1}^{k}\sum_{B'_1\cap B_0 \neq
\emptyset}\sum_{B'_1\cap B_1\neq \emptyset,B_1\nsubseteq
B'_1,\atop\left[\frac{t_1-t'_1}{E}\right]
\geq t}\leb_D(Z_1(k-t,B_1))\\
&\leq& \sum_{t=1}^{k}\sum_{B'_1\cap B_0 \neq
\emptyset}\sum_{B'_1\cap B_1\neq \emptyset,B_1\nsubseteq
B'_1,\atop\left[\frac{t_1-t'_1}{E}\right] \geq
t}D_1(D_1 \sigma^{E/2})^{k-t}\leb_D(C(B_{1}))\\
&\leq& \sum_{t=1}^{k}\sum_{B'_1\cap B_0\neq \emptyset}D_1(D_1
\sigma^{{E}/{2}})^{k-t}C_2 (C_5+1)\rho_1\sigma^{{Et}/{2}}\leb_D(C(B'_1))\\
&\leq& D_1(D_1\sigma^{{E}/{2}})^{k} C_2^2(C_5+1)\rho_1\rho_2
\sum_{t=1}^{k}D_1^{-t}\leb_D(C(B_{0})).
\end{eqnarray*}
Thus, taking  $D_1>0$ large enough so that $$
\frac{C_2^2(C_5+1)\rho_1\rho_2D_1^{-1}}{1-D_1^{-1}}\leq1$$
we finish the proof.
\end{proof}

Now set for positive  integers $k$ and $N$
\begin{equation}\label{eq.ze2}
Z_2(k,N) = \left\{ x\in\Delta_N\, |\,\exists B_1\varsupsetneq B_2
\ldots\varsupsetneq B_k\text{ with $t_1<\dots< t_k\leq N$ and $x\in
\bigcap_{i=1}^kB_i$} \right\}.
\end{equation}

\begin{Lemma}\label{Z_2}
There exists $\lambda_2<1$ such that for all $N\ge 1$ and $1\le k\le
N$
$$\leb_D(Z_2(k,N))\leq \lambda_2^k\leb_D(\Delta_0).$$
\end{Lemma}

\begin{proof}
We assume $N$ is fixed in this proof and simply write $Z_2(k)=Z_2(k,N)$. We shall prove
that the conclusion of the lemma holds with $\lambda_2=\frac{D_1}{D_1+1}$.
Using that $C_5+1\le D_1$, by Corollary~\ref{l.B}  we have for each possible $ B$ 
\begin{equation}\label{eq.Z_2} \leb_D(B) \leq D_1\leb_D(C(B)).
\end{equation}
We define $\mathcal {Q}_1$ as the class of sets $B$ with $t(B)\leq
N$ and  not contained in any other $B'$s. Consider
$\mathcal{Q}_2$ as the class of sets $B\notin \mathcal Q_1$ with
$t(B)\leq N$ which are included  in elements of~$\mathcal{Q}_1$ and not contained in any other $B$'s.
 We proceed inductively. Notice that this process
must stop in a finite number of steps because we always take
$t(B)\le N$. We say that an element in $\mathcal{Q}_i$ has
\emph{rank}~$i$.

Let now $$G_k=\bigcup_{i=1}^{k}\bigcup_{B\in \mathcal Q_k} C(B),$$
and
$$\tilde{Z}_2(k)=\left(\bigcup_{B\in{\mathcal{Q}_k}}B\right)\setminus G_k.$$
Now we prove that $Z_2(k)\subset\tilde Z_2(k)$. Given $x\in
Z_2(k)$, we have $x\in B_1 \cap\ldots\cap B_k\cap\Delta_N$ with
$B_1\supsetneq B_2 \ldots\supsetneq B_k$ and $t(B_k)\leq N$. We
clearly have that $B_k$ is of rank $r\geq k$. Take $B'_1\supsetneq
B'_2\ldots\supsetneq B'_{r-1}\supsetneq B'_r$ a sequence with $B'_i
\in \mathcal Q_i$ and $B'_r=B_k$. In particular, $x\in B'_i$ for
$i=1,\dots,k$, and so $x\in \bigcup_{B\in\mathcal Q_k}B$. On the
other hand, since $x\in\Delta_N$ and $G_k \cap\Delta_N=\emptyset$,
we get $x \notin G_k$. So $x\in \tilde{Z}_2(k)$.

Now we deduce the relation between $\leb_D(\tilde{Z}_2(k+1))$ and
$\leb_D(\tilde{Z}_2(k))$, in such a way that we may estimate
$\leb_D(\tilde{Z}_2(k))$. Take $B \in \mathcal Q_{k+1}$. Let $B'$ be
an element of rank $k$ containing $B$. As the cores are pairwise
disjoint by nature, $C(B)\cap G_k=\emptyset$. We obtain $C(B)\subset
B'\setminus G_k\subset \tilde{Z}_2(k)$. By definition $C(B)\subset
G_{k+1}$, thus $C(B)\cap
\tilde{Z}_2(k+1)=\emptyset$. This means that $C(B)\subset
\tilde{Z}_2(k)\setminus \tilde{Z}_2(k+1)$. Finally, by
\eqref{eq.Z_2},
\begin{eqnarray*}
  \leb_D(\tilde{Z}_2(k+1)) &\leq &\sum_{B\in
\mathcal{Q}_{k+1}}\leb_D(B)\\ & \le &D_1\sum_{B\in\mathcal{Q}_{k+1}}\leb_D(C(B))\\
                            & \le &D_1\leb_D(\tilde{Z}_2(k)\setminus \tilde{Z}_2(k+1))
\end{eqnarray*}
since the $C(B)$ are pairwise disjoint; recall Remark~\ref{re.cores}. Then, we obtain
\begin{eqnarray*}
  (D_1+1)\leb_D(\tilde{Z}_2(k+1)) &\le& D_1\leb_D(\tilde{Z}_2(k+1))+D_1\leb_D(\tilde{Z}_2(k)\setminus \tilde{Z}_2(k+1)) \\
    &=& D_1\leb_D(\tilde{Z}_2(k)).
\end{eqnarray*}
It yields $\leb_D(\tilde{Z}_2(k))\leq
\left(\frac{D_1}{D_1+1}\right)^k\leb_D(\Delta_0)$ by induction. Since $Z_2(k)\subset \tilde{Z}_2(k)$, the same
inequality holds for $Z_2(k)$. 
\end{proof}

The previous two lemmas  are enough for us to establish the desired estimate on the $\leb_D$ measure of $Z(k,N)$ in the next lemma. The idea is to  divide $N$ into blocks of length $E$ and choosing properly certain instants from each block. Depending on the number of such instants, we shall apply either Lemma~\ref{Z_1} or Lemma~\ref{Z_2}  to yield the conclusion.

\begin{Lemma}\label{Z}
There are  $D_3>0$ and $0<\lambda_3<1$
such that for all 
 $1\leq k \leq
N$,
$$\leb_D(Z(k,N))\leq D_3{\lambda_3^k}\leb_D(\Delta_0).$$
\end{Lemma}

\begin{proof}
Choose $E$ large enough such that $D_1\sigma^{E/2}<1$; recall $D_1$ in 
Lemma~\ref{Z_1}. We write $N=rE+s$ with $s<E$. Given any
 $x\in Z(k,N)$, choose the instants $t_1<\ldots<t_k$ as
in the definition of $Z(k,N)$. For $0\leq u <r$, take from each
interval $[uE,(u+1)E)$ the first  $t_i \in \{t_1,\dots,t_k\}$ (if there is at least one). Denote that subsequence of $t_i$'s by  $t_{1'}<\dots<t_{k'}$. Since
$t_1<\dots<t_k\le N$, we have $k'\ge [\frac kE]$, which means that
$Ek'+E\geq k$. Keeping only the instants with odd indexes, we get a
sequence of instants $u_1<\ldots<u_\ell$ with $2\ell\geq k'$, and necessarily
$\ell\geq \frac{k-E}{2E}$. Moreover, we have $u_{i+1}-u_i\geq E$ for $1\le
i\le\ell$, by construction.
According to our construction process, we know that associated to the instant $u_i$
there must be some set $B_i$ such that  $x\in B_i$,
for $1\le i\le\ell$.
Define $$I=\{1\leq i\leq\ell, B_i\subset B_1\cap\dots\cap B_{i-1}\}\quad\text{and}\quad J=[1,\ell]\setminus I.$$
Now we split the proof according to the following possible cases:

If $\# I\geq \ell/2$, we keep only the
elements with indexes in $I$. We necessarily have $x\in Z_2(\ell/2,N)$; recall \eqref{eq.ze2}. Then, using  Lemma \ref{Z_2} we see that
$Z_2(\ell/2,N)$ has $\leb_D$ measure exponentially small  in $\ell$ (hence in $k$), which gives the result in this case.

If $\# I\leq \ell/2$, then $\# J \geq \ell/2$. Set
$j_0=\sup J$ and $i_0=\inf\{i<j_0,B_{j_0}\nsubseteq
 B_i\}.$ Next set $j_1=\sup\{j\leq i_0,j\in J\}$ and $i_1=\inf\{i<j_1,B_{j_1}\nsubseteq
 B_i\}$. Proceeding inductively, the process must necessarily stop at some step  $i_n$.
 Then $J\subset\cup_{s=0}^{n}(i_s,j_s]$, by construction. We obtain $\sum_{s=0}^{n}(j_s-i_s)\geq\# J\geq
 \ell/2$, which shows that $$\sum_{s=0}^{n}\left[\frac{t(B_{j_s})-t(B_{i_s})}{E}\right]=\sum_{s=0}^{n}\left[\frac{u_{j_s}-u_{i_s}}{E}\right]\geq
 \ell/2,$$ since $|u_j-u_i|\ge E(j-i)$ by the constructing process. Hence $x \in
Z_1(\ell/2,B_{i_n})$ with the sequence
$B_{i_n},B_{i_n},B_{j_n},\dots,B_{i_0},B_{j_0}$. As the cores of these sets are
pairwise disjoint, we use 
Lemma~\ref{Z_1} and, summing over all the possible $B_{i_n}'$s,   we get   the result also in this case.
\end{proof}

Finally, we deduce an estimate on the $\leb_D$ measure of the sets $Y(n_1,\ldots,n_p,B_1)$ as in~\eqref{eq.Yi}. Recall that $P$ has been introduce in Lemma~\ref{l.P}.

\begin{Lemma}\label{Y}
There is  $D_4>0$ 
such that for all $n_1,\ldots,n_p>P$ and $B_1$
$$\leb_D(Y(n_1,\ldots,n_p,B_1))\leq
D_4(D_4\sigma^{n_1 /2})\cdots(D_4\sigma^{n_p /2})\leb_D (C(B_1)).
$$
\end{Lemma}
\begin{proof}
The proof is by induction on  $p$. Taking
$D_4>{C_5}^{1/2}$ (recall the constant $C_5$ in Proposition~\ref{d.prop.Sn}), we immediately get
the result for $p=1$.
Now suppose $p>1$ and let $x\in Y(n_1,\ldots,n_p,B_1)$. Then there exists
$B_2=B_{t_2}^{t_2}(x_2)$ constructed
 at an instant $t_2>t_1$ such that  $x\in Y(n_2,\ldots,n_p, B_2)$.
 By Lemma~\ref{l.P} we have
$B_{t_2+P}^{t_2}(x_2)\cap B_{t_2+P}^{t_1}(x_1)=\emptyset.$ But for
all $1\leq i \leq s$, we have $x \in
 B_{t_i+n_i}^{t_i}(x_i).$
So, $t_1+n_1< t_2+P,$ i.e. $t_2-t_1>n_1-P$. By the uniform expansion at hyperbolic times,
we get
$$\diam_{f^{t_1}(D)}(f^{t_1}(B_2))\leq \sigma^{\frac{t_2-t_1}{2}}
\diam_{f^{t_2}(D)}(f^{t_2}(B_2))\leq
6\delta'_1\sigma^{\frac{n_1-P}{2}}.
$$
On the other hand, setting
$Q=f^{t_1}(C(B_1))$, we have $\dist_{f^{t_1}(D)}(f^{t_1}(x),\partial Q)\leq
2\delta'_1\sigma^{\frac{n_1}{2}}$ when $x\in
B_{t_1+n_1}^{t_1}(x_1)\cap B_2$.
Thus, there is some constant $D_5>0$ such that
$$f^{t_1}(B_2)\subset \mathcal{C}:=\left\{y\,\mid \, \dist_{f^{t_1}(D)}(y,\partial Q)\leq
D_5\sigma^\frac{n_1}{2}\right\}.$$ By the induction hypothesis we have
 \begin{equation*}
 \leb_{D} (Y(n_2,\ldots,n_p,B_2))
  \leq D_4(D_4\sigma^{n_2 /2})\ldots(D_4\sigma^{n_p
  /2})\leb_{D}(C(B_2)),
 \end{equation*}
which together with bounded distortion yields
 \begin{equation*}
 \leb_{f^{t_1}(D)} (f^{t_1}(Y(n_2,\ldots,n_p,B_2)))
  \leq C_2D_4(D_4\sigma^{n_2 /2})\ldots(D_4\sigma^{n_p
  /2})\leb_{f^{t_1}(D)}(f^{t_1}(C(B_2))).
 \end{equation*}
By Remark~\ref{re.cores} and Lemma~\ref{l.in V_n}, the possible cores $C(B_2)$'s are pairwise disjoint  and all contained in
 $V_{t_1}(x_1)$.
So, the sets $f^{t_1}(C(B_2))$ are still pairwise disjoint and all contained in the annulus $\mathcal {C}$. Hence
 \begin{align*}
 \leb_{f^{t_1}(D)} (f^{t_1}(Y(n_1,\ldots,n_p,B_1))) &\leq
 \sum_{B_2}\leb_{f^{t_1}(D)}(f^{t_1}(Y(n_2,\ldots,n_p,B_2)))\\
&\leq {C_2}{D_4}(D_4\sigma^{n_2 /2})\dots(D_4\sigma^{n_p
/2})\sum_{B_2}\leb_{f^{t_1}(D)}
 (f^{t_1}(C(B_2)))\\
 &\leq {C_2}{D_4}(D_4\sigma^{n_2 /2})\dots(D_4\sigma^{n_p /2})\leb_{f^{t_1}(D)}(\mathcal
 {C}).
 \end{align*}
Similarly to \eqref{bdy2}, we have for some $\rho>0$
$$\leb_{f^{t_1}(D)}(\mathcal{C})\leq
 \rho\sigma^{n_1 /2}\leb_{f^{t_1}(D)}(Q),$$ where $Q=f^{t_1}(C(B_1))$.
 Then,
$$\leb_{f^{t_1}(D)}(f^{t_1}(Y(n_1,\dots,n_p,B_1)))\leq C_2\rho\sigma^{n_1 /2}(D_4\sigma^{n_2 /2})\ldots(D_4\sigma^{n_p /2})\leb_{f^{t_1}(D)}
 (Q),$$
 which together with bounded distortion yields
 $$\leb_D(Y(n_1,\dots,n_p,B_1))\leq C_2^2\rho(D_4\sigma^{n_1 /2})(D_4\sigma^{n_2 /2})\dots(D_4\sigma^{n_p /2})\leb_D
 (C(B_1)).$$
Taking $D_4\geq C_2^2\rho$, we finish the proof.
\end{proof}

Now we are ready to complete the proof of  Proposition~\ref{X}.
Take $P'\geq P$ (recall $P$ in Lemma~\ref{l.P}) so that
\begin{equation}\label{eq.ineq}
\sigma^{1/2}+D_4\sigma^{P'/2}<1.
\end{equation}
%
As shown in~\eqref{Z12}, we have
$$X(k,N) \subset
\bigcup_{B_1}\bigcup_{n_1,\ldots,n_p\geq
P',\atop\sum{n_i}\geq\frac{k}{2}}Y(n_1,\ldots,n_p,B_1)\cup{Z\left(\left[\frac{k}{2P'}\right],N\right)}.
$$
It follows from Lemma~\ref{Z} and Lemma~\ref{Y} that
$$\leb_D(X(k,N))\leq \sum_{B_1}\sum_{n_1,\ldots,n_p\geq
P',\atop\sum{n_i}\geq\frac{k}{2}} D_4(D_4\sigma^{n_1
/2})\ldots(D_4\sigma^{n_p
/2})\leb_D(C(B_1))+D_3\lambda_3^{\frac{k}{2P'}}\leb_D(\Delta_0).
$$
We have $\sum_{B_1}\leb_D(C(B_1))\leq \leb_D(\Delta_0)<\infty$, because the
cores $C(B_1)$ are pairwise disjoint. We are left to show that the sum
$$\sum_{n_1,\ldots,n_p\geq
P',\atop\sum{n_i}\geq\frac{k}{2}}(D_4\sigma^{n_1
/2})\ldots(D_4\sigma^{n_p /2})
$$
is exponentially small in $k$.
We use the generating series
$$\sum_n\sum_{n_1,\ldots,n_p\geq
P',\atop\sum{n_i}=n}(D_4\sigma^{n_1 /2})\ldots(D_4\sigma^{n_p
/2})z^n=\sum_{p=1}^{\infty}\left(D_4\sum_{n=P'}^{\infty}\sigma^{n/2}z^n
 \right)^p=\frac{D_4\sigma^{P'/2}z^{P'}}{1-{\sigma^{1/2}}z-D_4\sigma^{P'/2}z^{P'}}.
$$
Under condition \eqref{eq.ineq}, the function above has no  pole in a neighborhood of the unit disk
in~$\mathbb{C}$. Thus, its coefficients decay
exponentially fast: there are constants $D_6>0$ and $\lambda_6<1$
such that
$$\sum_{n_1,\ldots,n_p\geq
P',\atop\sum{n_i}=n}(D_4\sigma^{n_1
/2})\ldots(D_4\sigma^{n_p /2})\leq
D_6\lambda_6^n.
$$
Then we sum over $n\geq k/2$ and $B_1$ to obtain constants $D_0>0$ and $0<\lambda_0<1$ such that
$$\leb_D(X(k,N))\leq D_0\lambda_0^k\leb_D(\Delta_0),$$
which gives Proposition~\ref{X}.

\section{Gibbs-Markov-Young structure}

In this section we construct the GMY structure given by Theorem~\ref{t:Markov towers}.

\subsection{Product structure}\label{s.product}



Consider the center-unstable disk $\Delta_0\subset D$ of Section~\ref{s.partition} and the
($\leb_D$ mod 0) partition $\cp$ of $\Delta_0$. We define
$$\Gamma^s=\left\{ W^s_{\delta_s}(x):\,x\in \Delta_0\right\}$$
and  the family of unstable leaves
 $\Gamma^u$ as the set of all local unstable leaves  $u$-crossing $\mathcal C^0$; recall Definition~\ref{d.ucross} and~\eqref{eq.C0}. Clearly
$\Gamma^u$ is nonempty because $\Delta_0\in \Gamma^u$. We need to see  that $\cup\gamma^u$ is compact. By the domination
property and Ascoli-Arzela Theorem, any limit leaf $\Delta_\infty$
of leaves in $\Gamma^u$ is a center-unstable disk $u$-crossing $\mathcal C^0$. Hence
$\Delta_\infty\in\Gamma^u$, by definition of $\Gamma^u$, and so $\cup\gamma^u$ is compact.

The \emph{s-subsets} are defined in the following way: given $\omega\in \mathcal P$, consider
 $$\cc(\omega)=\bigcup_{x\in\omega}W^s_{\delta_s}(x).$$
The pairwise disjoint $s$-subsets $\Lambda_1,\Lambda_2,\dots$ are
precisely the sets $$\{\cc(\omega)\cap(\cup\gamma^u):  \omega\in\cp\}.$$
Then we should check that $f^{R_i}(\Lambda_i)$ is \emph{u-subset}.
Given an element $\omega\in\cp$, by construction there is some
$R(\omega)\in\NN$ such that $f^{R(\omega)}(\omega)$ is a
center-unstable disk $u$-crossing~$\cc^0$. 
Since by construction
$f^{R(\omega)}(\omega)$ intersects
 $W^s_{\delta_s/4}(p)$, then according to the choice of
 $\delta_0$ and the invariance of the stable foliation, we have
 that each element of $f^{R(\omega)}(\cc(\omega)\cap\Gamma^u)$ must $u$-cross
 $\cc^0$ and is contained in the $\lambda^{R(\omega)}\delta_s$ height
 neighborhood of $f^{R(\omega)}(\omega)$. Ignore the difference caused by the angle. We can say it is contained in $\mathcal C^0$. So, that is a \emph{u-subset}.

In the sequel we prove that the set $\Lambda=(\cup\gamma^u)\cap(\cup\gamma^s)$
 with a \emph{product structure} has indeed a \emph{GMY structure}.
Observe that  $\Lambda$ coincides with the union of the
leaves in $\Gamma^u$ 
and properties $(\bf P_0)$--$(\bf P_2)$ are naturally satisfied. In the next two subsections we prove properties $(\bf P_3)$ and   $(\bf P_4)$. 



\subsection{Uniform expansion and bounded distortion}
Property $(\bf P_3)$(a)  follows from the next result.

 \cle\label{backcontraction} There is $C>0$ such that, given   $\omega\in\cp$  and $\gamma\in\Gamma^u$,
we have for all $1\le k \le R(\omega)$ and all $x,y\in
\cc(\omega)\cap\gamma$
 $$
\dist_{f^{R(\omega)-k}(\cc(\omega)\cap\gamma)}(f^{R(\omega)-k}(x),f^{R(\omega)-k}(y))
\le
C\sigma^{k/2}\dist_{f^{R(\omega)}(\cc(\omega)\cap\gamma)}(f^{R(\omega)}(x),f^{R(\omega)}(y)).
 $$
 \fle

\dem 

Let $\omega$ be an element of the partition $\cp$ constructed  in 
Section \ref{s.partition}. There is necessarily  a point $x \in D$ with
$\sigma$-hyperbolic time $n(\omega)$ satisfying $R(\omega)-N_0\le
n(\omega)\le R(\omega)$.
Since we take $\delta_s,\delta_0<\delta_1/2$, by \eqref{e.delta1}, $n(\omega)$ is  a $\sqrt\sigma$-hyperbolic
time for every point in $\cc(\omega)\cap\gamma$. Recalling \eqref{contra}, we obtain that for all $1\le k \le n(\omega)$ and
all $x,y\in \cc(\omega)\cap\gamma$
 $$
\dist_{f^{n(\omega)-k}(\cc(\omega)\cap\gamma)}(f^{n(\omega)-k}(x),f^{n(\omega)-k}(y))
\le
\sigma^{k/2}\dist_{f^{n(\omega)}(\cc(\omega)\cap\gamma)}(f^{n(\omega)}(x),f^{n(\omega)}(y)).
 $$
Considering $R(\omega)-n(\omega)\le N_0$, we take $C$ depending only on $N_0$ and the  derivative of~$f$, then we get the result. \cqd

Property $(\bf P_3)$(b) follows from Proposition~\ref{c.curvature}
together with Lemma~\ref{backcontraction} as in \cite[Proposition
2.8]{ABV}. We prove it here for the sake completeness.

 \cle
\label{l.boundeddistortion} There is $\bar{C}>0$ such that, for all
$x, y \in \Lambda_i$ with $y\in\gamma^u(x)$, we have
$$
\log\frac{\det D(f^{R_i})^u(x)}{\det D(f^{R_i})^u(y)}\leq \bar{C}
\dist(f^{R_i}(x),f^{R_i}(y))^\zeta.
$$
\fle

\dem For $0\leq k< R_i$ and $y \in \gamma^u(x)\in \Gamma^u$, we
set $J_k(y)=\log |\det Df^u(f^k(y))|$ as in the last item of Proposition~\ref{c.curvature}. Then,
$$
\log\frac{\det D(f^{R_i})^u(x)}{\det D(f^{R_i})^u(y)}=
\sum_{k=0}^{R_i-1}(J_k(x)-J_k(y)) \leq \sum_{k=0}^{R_i-1}L_1
\dist_D(f^k(x), f^k(y))^\zeta.
$$
By Proposition~\ref{c.curvature}, the sum of $\dist_D(f^k(x),
f^k(y))^\zeta$ over $0\leq k \leq R_i$ is bounded by
$$\dist_D(f^{R_i}(x), f^{R_i}(y))^\zeta / (1-\sigma^{\zeta/2}).$$ Take $\bar C= L_1(1-\sigma^{\zeta/2})$ to get the result. \cqd


\subsection{Regularity of the foliations}
\label{sec.regularity}
Property ($\bf P_4$) has already been proved in \cite{AP3}. 
This is standard  for uniformly hyperbolic attractors and the ideas can be adapted to the partially hyperbolic setting. Property
($\bf P_4$)(a) follows from the next result whose proof may be found in \cite[Corollary 3.8]{AP3}.

\cpr\label{pr.produtorio} There are $C>0$ and $0<\beta<1$ such that
for all $y\in\gamma^s(x)$ and $ n\ge 0$
 $$\displaystyle
 \log \prod_{i=n}^\infty\frac{\det Df^u(f^i(x))}{\det Df^u(f^i(y))}\le C\beta^{n}.
 $$
\fpr

For ($\bf P_4$)(b) we need to introduce some useful notion. 

\cd\label{de.abscont}
Given $N$ and $G$ 
submanifolds of~$M$, we say that $\phi: N\to G$ is \emph{absolutely continuous} if it is an
injective map for which there exists $J:N\to\RR$, called the
\emph{Jacobian} of $\phi$, such that
 $$
 \leb_G(\phi(A))=\int_A Jd\leb_N.
 $$
 \fd
 
Property ($\bf P_4$)(b) follows from the next result  whose proof is given in \cite[Proposition~3.9]{AP3}.

\cpr\label{pr.regulstable} Given $\gamma,\gamma'\in\Gamma^u$, define
$\phi\colon\gamma'\to\gamma$  by $\phi(x)=\gamma^s(x)\cap \gamma$.
Then $\phi$ is absolutely continuous and  the  Jacobian  of $\phi$
is given by
        $$
        J(x)=
        \prod_{i=0}^\infty\frac{\det Df^u(f^i(x))}{\det
        Df^u(f^i(\phi(x)))}.$$
\fpr

We have from Proposition~\ref{pr.produtorio} that this
infinite product converges uniformly.

\section{Application}\label{s.appli}

Here we present a open  class of partially hyperbolic diffeomorphisms 
whose center-unstable direction is non-uniformly expanding at Lebesgue almost everywhere in $M$ and, for any center-unstable disk $D$, we have  $\leb_D\{\mathcal E>n\}$ is exponentially small. 
This example was introduced in \cite[Appendix]{ABV} and we sketch  below the main steps of its description.


We consider a linear Anosov diffeomorphism $f_0$ on the $d$-dimensional torus $M=T^d$, $d\ge 2$, with a  hyperbolic splitting $TM=E^u \oplus E^s$. Let $V\subset M$ be some small compact domain, such that for $\pi:\mathbb R^d\to T^d$ the canonical projection, there exist unit open cubes $K^0$, $K^1$ in $\mathbb R^d$ such that $V\subset\pi(K^0)$ and $f_0(V)\subset\pi(K^1)$.  Let $f$ be a diffeomorphism on $T^d$ such that:
\begin{enumerate}
\item $f$ has invariant cone fields $C^{cu}$ and $C^s$ which are with small width $\alpha>0$ and contain, respectively, the unstable bundle $E^u$ and the stable bundle $E^s$ of the Anosov diffeomorphism $f_0$;
\item $f^{cu}$ is \emph{volume expanding everywhere}: there is $\sigma_1>0$ such that  $|\det(Df|T_x D^{cu})|>~\sigma_1$ for any $x\in M$ and any disk $D^{cu}$ through $x$ tangent to the center-unstable cone field $C^{cu}$;
\item $f$ is $C^1$-close to $f_0$ in the compliment of $V$, so that $f^{cu}$ is \emph{expanding outside $V$}: there is $\sigma_2<1$ satisfying $\|(Df|T_x D^{cu})^{-1}\|<\sigma_2$ for $x\in M\setminus V$ and any disks $D^{cu}$ tangent to $C^{cu}$;
\item $f^{cu}$ is \emph{not too contracting} on $V$: there is small $\delta_0>0$ satisfying $\|(Df|T_x D^{cu})^{-1}\|<1+\delta_0$ for any $x\in V$ and any disks $D^{cu}$ tangent to $C^{cu}$.
\end{enumerate}

For example, if $f_1:T^d \to T^d$ is a diffeomorphism satisfying itens (1), (2) and (4) above and coinciding with $f_0$ outside $V$, then any $f$ in a $C^1$ neighborhood of $f_1$ satisfies all the conditions (1)-(4). The $C^1$ open classes of transitive non-Anosov diffeomorphisms given in \cite[Section 6]{BV}, and also other robust examples from \cite{Man}, are constructed in this way and they satisfy: both these diffeomorphisms and their inverse satisfy conditions (1)-(4) above.

Next we show that any $f$ satisfying (1)-(4) is non-uniformly expanding along the $cu$-direction on a full Lebesgue set of points in $M$.
 Let $B_1,\ldots, B_p, B_{p+1}=V$ be any partition of $T^d$ into small subsets such that there exist open cubes $K_i^0$ and $K_i^1$ in $\mathbb R^d$ for which $$B_i\subset \pi(K_i^0) \qand f(B_i)\subset\pi(K_i^1).$$
Let $\mathcal F_0^u$ be the unstable foliation of $f_0$ and 
let us fix any small disk $D$ contained in a leaf of~$\mathcal F_0^u$.
Using the same arguments in the proof of \cite[Lemma A.1]{ABV} we deduce the next result.

\cle\label{l.eg} 
There exist
$\theta>0$ such that the orbit of Lebesgue almost every $x\in D$
spends a fraction $\theta$ of the time in $B_1\cup\ldots\cup B_p$:
\begin{equation*}
\#\{0\leq j<n: f^j(x)\in B_1\cup\ldots\cup B_p\}\geq\theta
n\end{equation*} for every large $n$. 
\fle
Hence, $\leb_{D}$-almost every point $x\in D$ spends a positive fraction $\theta$ of time outside the domain $V$. 
 Then by itens (3) and (4) above, there exists $c_0>0$ such that
$$\limsup_{n\rightarrow +\infty}\frac{1}{n}\sum_{j=0}^{n-1}\log\|(Df \mid E^{cu}_{f^j(x)})^{-1}\|\leq -c_0$$ for $\leb_{D}$-almost every point $x\in D$. Moreover, there exists a constant $c>0$ such that
 $$\leb_{D}\{\mathcal E>n\}=\mathcal O(e^{-cn});$$
see the Claim in~\cite[p.396]{ABV}. Furthermore, as $D$ is an arbitrary $cu$-disk, $f$ is non-uniformly expanding along the $cu$-direction on a full Lebesgue set of points in $M$.

%
%
%

\end{document}